\documentclass[a4paper,10pt]{article}

\usepackage[latin1]{inputenc}
\usepackage{amsfonts}
\usepackage{amsmath}
\usepackage{amssymb}
\usepackage{amsthm}
\usepackage[T1]{fontenc}
\usepackage[dvips]{graphicx}
\usepackage{pstricks,pst-node,pst-text,pst-3d,epsfig,pifont}
\usepackage{subfigure}

\author{Robert J Taggart}
\title{Inhomogeneous Strichartz estimates}
\date{28 February 2008}

\newtheorem{theorem}{Theorem}[section]
\newtheorem{corollary}[theorem]{Corollary}
\newtheorem{lemma}[theorem]{Lemma}
\newtheorem{proposition}[theorem]{Proposition}

\theoremstyle{definition}
\newtheorem{definition}[theorem]{Definition}

\theoremstyle{remark}

\newtheorem{remark}[theorem]{Remark}

\newcommand{\Hi}{\mathcal{H}}
\newcommand{\R}{\mathbb{R}}
\newcommand{\Z}{\mathbb{Z}}

\newcommand{\A}{\mathcal{A}}
\newcommand{\B}{\mathcal{B}}
\newcommand{\C}{\mathcal{C}}

\newcommand{\Q}{\mathcal{Q}}

\newcommand{\Bth}{\mathcal{B}_{\theta}}
\newcommand{\Bwth}{\mathcal{B}_{\wt{\theta}}}
\newcommand{\bs}{\backslash}
\newcommand{\norm}[1]{\left\Vert{#1}\right\Vert}
\newcommand{\wt}[1]{\widetilde{#1}}
\newcommand{\dd}{\mathrm{d}}
\newcommand{\dist}{\mathrm{dist}}
\newcommand{\st}[2]{L^{#1}(\R;#2)}
\newcommand{\stI}[2]{L^{#1}(I;#2)}
\newcommand{\stJ}[2]{L^{#1}(J;#2)}
\newcommand{\stn}[2]{_{\st{#1}{#2}}}
\newcommand{\stnI}[2]{_{\stI{#1}{#2}}}
\newcommand{\stnJ}[2]{_{\stJ{#1}{#2}}}

\newcommand{\bes}[3]{\dot{B}_{#2,#3}^{#1}}
\newcommand{\besn}[3]{_{\dot{B}_{#2,#3}^{#1}}}
\newcommand{\sob}[2]{\dot{H}_{#2}^{#1}}

\newcommand{\hsob}[1]{\dot{H}^{#1}}
\newcommand{\hsobn}[1]{_{\dot{H}^{#1}}}
\def\ip<#1,#2>{\left\langle#1,\,#2\right\rangle}
\def\nn<#1>{\left\langle#1\right\rangle}

\begin{document}

\maketitle

\begin{abstract}
\noindent We present abstract inhomogeneous Strichartz estimates for dispersive operators, extending previous work by M. Keel and T. Tao on the one hand, and generalising results of D. Foschi, M. Vilela, M. Nakamura and T. Ozawa on the other hand. It is shown that these abstract estimates imply new inhomogeneous Strichartz estimates for the wave equation and some Schr\"odinger equations involving potentials.
\end{abstract}

\section{Introduction}\label{s:intro}

This paper is concerned with a priori estimates for dispersive partial differential equations, expressed in norms given by
\[\norm{G}\stn{q}{\B}=\left(\int_{\R}\norm{G(t)}_{\B}^q\,\dd t\right)^{1/q},\]
where $\B$ is a Banach space. Consider, for example, the inhomogeneous Schr\"odinger initial value problem
\begin{equation}\label{IVP:schrodinger intro}
\begin{cases}
&iu'(t)+\Delta u(t)=F(t)\qquad\forall t\geq0\\
&u(0)=f,
\end{cases}
\end{equation}
whose formal solution $u$ is given by
\[u(t)=e^{it\Delta}f-i\int_0^te^{i(t-s)\Delta}F(s)\,\dd s\]
via Duhamel's principle. The seminal paper \cite{rS77} of R. Strichartz showed that if $u$ is a solution to (\ref{IVP:schrodinger intro}) in $n$ spatial dimensions and $q=r=2(n+2)/n$, then
\begin{equation}\label{est:rS77 schrodinger estimate}
\norm{u}\stn{q}{L^r(\R^n)}\lesssim\norm{f}_{L^2(\R^n)}+\norm{F}\stn{q'}{L^{r'}(\R^n)}
\end{equation}
whenever $f\in L^2(\R^n)$ and $F\in \st{q'}{L^{r'}(\R^n)}$. Since then, various authors (see especially \cite{GV85a}, \cite{kY87}, \cite{CW88} and \cite{KT98}) have published similar a priori estimates for solutions to Schr\"odinger's equation where the time exponent $q$ and the space exponent $r$ are unequal. Such estimates have proved fruitful for determining whether various semilinear Schr\"odinger equations are well-posed (see, for example, \cite{GV85a}, \cite{tK87} and \cite{CW90}).

By respectively taking $F$ and $f$ as $0$ in (\ref{IVP:schrodinger intro}), estimate (\ref{est:rS77 schrodinger estimate}) becomes
\begin{equation}\label{est:homogeneous schrodinger}
\norm{e^{it\Delta}f}\stn{q}{L^r(\R^n)}\lesssim\norm{f}_{L^2(\R^n)}
\end{equation}
and
\begin{equation}\label{est:inhomogeneous schrodinger}
\norm{\int_0^te^{i(t-s)\Delta}F(s)\,\dd s}\stn{q}{L^r(\R^n)}\lesssim\norm{F}\stn{q'}{L^{r'}(\R^n)}.
\end{equation}
The first of these estimates is called an \textit{homogeneous Strichartz estimate} (since it solves the homogeneous Schr\"odinger equation with initial data $f$) and the second is known as an \textit{inhomogeneous Strichartz estimate} (since it solves the inhomogeneous Schr\"odinger equation with forcing term $F$ and zero initial data). The problem of finding all possible exponents pairs $(q,r)$ such that (\ref{est:homogeneous schrodinger}) is valid has been completely solved. That is, (\ref{est:homogeneous schrodinger}) holds if and only if
\begin{equation}\label{eq:schrodinger admissible}
q\in[2,\infty],\qquad\frac{1}{q}+\frac{n}{2r}=\frac{n}{4}\qquad\mbox{and}\qquad(q,r,n)\neq(2,\infty,2)
\end{equation}
(see \cite{KT98} and the references therein). The corresponding problem for the inhomogeneous Strichartz estimate (\ref{est:inhomogeneous schrodinger}) remains open. It is known that if the exponent pairs $(q,r)$ and $(\wt{q},\wt{r})$ satisfy (\ref{eq:schrodinger admissible}) then the inhomogeneous estimate (\ref{est:inhomogeneous schrodinger}) is also valid. However, it was observed by T. Cazenave and F. Weissler \cite{CW92} and T. Kato \cite{tK94} that there are exponent pairs $(q,r)$ for which the inhomogeneous Strichartz estimate holds but the homogeneous estimate fails. Using the techniques introduced in \cite{KT98}, D. Foschi \cite{dF05a} and M. Vilela \cite{mV07} independently obtained what is currently the largest known range of exponent pairs $(q,r)$ and $(\wt{q},\wt{r})$ for which the inhomogeneous Strichartz estimate (\ref{est:inhomogeneous schrodinger}) for the Schr\"odinger equation is valid.

Similar remarks may be made for the wave equation. The precise set of Lebesgue spacetime exponents for which the homogeneous Strichartz estimate for the wave equation is known (see \cite{LS95}, \cite[Section 1]{KT98} and the references therein) while a complete description for the set of exponents for which the inhomogeneous estimate is valid remains open (see the early work of D. Oberlin \cite{dO89} and J. Harmse \cite{jH90} and more recent advances by Foschi \cite{dF05a}). It must also be noted that, previous to the work of Foschi, a large set of exponents for inhomogeneous Strichartz-type estimates for solutions of the Klein--Gordon equation were obtained by M. Nakamura and T. Ozawa \cite{NO01}.

In this paper, results of \cite{NO01}, \cite{mV07} and especially \cite{dF05a} are generalised to the abstract setting introduced in \cite{KT98}, thus enabling us to find new inhomogeneous Strichartz-type estimates for other dispersive equations, including the wave equation and Schr\"odinger equations with potential. Suppose that $\Hi$ is a Hilbert space with inner product $\ip<\,\cdot\,,\cdot\,>$, $(\B_0,\B_1)$ is a Banach interpolation couple and $\sigma>0$. Suppose also that for each time $t$ in $\R$ we have an operator $U(t):\Hi\to\B_0^*$. Its adjoint $U(t)^*$ is an operator from $\B_0$ to $\Hi$. We will assume that the family $\{U(t):t\in\R\}$ satisfies the \textit{energy estimate}
\begin{equation}\label{est:energy}
\norm{U(t)f}_{\B_0^*}\lesssim\norm{f}_{\Hi} \qquad\forall f\in\Hi\quad\forall t\in\R,
\end{equation}
and the \textit{dispersive estimate}
\begin{equation}\label{est:untrunc_decay}
\norm{U(s)U(t)^*g}_{\B_1^*}\lesssim|t-s|^{-\sigma}\norm{g}_{\B_1}
\qquad\forall g\in \B_1\cap\B_0 \quad \forall \mbox{ real } s\neq t.
\end{equation}
Using the energy estimate, we consider the operator $T:\Hi\to \st{\infty}{\B_0^*}$, defined by the formula
\[Tf(t)=U(t)f\qquad\forall f\in\Hi\quad\forall t\in\R.\]
Its formal adjoint $T^*:\st{1}{\B_0}\to\Hi$ is given by the $\Hi$-valued integral
\[T^*F=\int_{\R}U(s)^*F(s)\,\dd s.\]
The composition $TT^*:\st{1}{\B_0}\to\st{\infty}{\B_0^*}$, given by
\[TT^*F(t)=\int_{\R}U(t)U(s)^*F(s)\,\dd s,\]
can be decomposed as the sum of retarded and advanced parts, respectively given by
\[(TT^*)_RF(t)=\int_{s<t}U(t)U(s)^*F(s)\,\dd s\]
and
\[(TT^*)_AF(t)=\int_{s>t}U(t)U(s)^*F(s)\,\dd s.\]
In applications (see Sections \ref{s:schrodinger} and \ref{s:wave} for examples), $\{U(t):t\in\R\}$ is the evolution family associated to a homogeneous differential equation, $T$ solves the initial value problem of the homogeneous equation and $(TT^*)_R$ solves the corresponding inhomogeneous problem with zero initial data. Hence, if $\Bth$ denotes the real interpolation space $(\B_0,\B_1)_{\theta,2}$ whenever $\theta\in[0,1]$, then corresponding homogeneous and inhomogeneous Strichartz estimates are given by
\begin{equation}\label{eq:str1}
\norm{Tf}\stn{q}{\Bth^*}\lesssim\norm{f}_{\Hi}\qquad\forall f\in\Hi
\end{equation}
and
\begin{equation}\label{eq:str3}
\norm{(TT^*)_RF}\stn{q}{\Bth^*}\lesssim\norm{F}\stn{\wt{q}'}{\Bwth}
\qquad\forall F\in\st{\wt{q}'}{\Bwth}\cap\st{1}{\B_0}.
\end{equation}
The following theorem of M. Keel and T. Tao gives conditions on the exponent pairs $(q,\theta)$ and $(\wt{q},\wt{\theta})$ such that these abstract Strichartz estimates hold.

\begin{definition}
Suppose that $\sigma>0$. We say that a pair of exponents $(q,\theta)$ is \textit{sharp $\sigma$-admissible} if $(q,\theta,\sigma)\neq(2,1,1)$, $2\leq q\leq \infty$, $0\leq\theta\leq1$ and  $\frac{1}{q}=\frac{\sigma\theta}{2}$.
\end{definition}

\begin{theorem}[Keel--Tao, Theorem 10.1 \cite{KT98}]\label{th:KT98} Suppose that $\sigma>0$. If $\{U(t):t\in\R\}$ satisfies the energy estimate (\ref{est:energy}) and the dispersive estimate (\ref{est:untrunc_decay}) then the Strichartz estimates (\ref{eq:str1}) and (\ref{eq:str3}) hold for all sharp $\sigma$-admissible pairs $(q,\theta)$ and $(\wt{q},\wt{\theta})$.
\end{theorem}

The range of exponents given by Theorem \ref{th:KT98} for which the homogeneous estimate (\ref{eq:str1}) is valid cannot be improved. (One can show this by considering the case when $U(t)=e^{it\Delta}$, $\sigma=n/2$, $\Hi=L^2(\R^n)$ and $(\B_0,\B_1)=(L^2(\R^n),L^1(\R^n))$, which corresponds to setting of the Schr\"odinger equation; see, for example, \cite[Section 8]{KT98}.) However, as noted in \cite{KT98}, the range of exponents given by Theorem \ref{th:KT98} for the inhomogeneous estimate (\ref{eq:str3}) is suboptimal. One of the aims of this paper is to extend this range. This has already been achieved by D. Foschi \cite[Theorem 1.4]{dF05a} for the special case when $(\B_0,\B_1)=(L^2(X),L^1(X))$. The following theorem, introduced after Definition \ref{def:sigma-acceptable}, generalises Foschi's result and is the main theorem of our paper.

\begin{definition}\label{def:sigma-acceptable}
Suppose that $\sigma>0$. We say that a pair $(q,\theta)$ of exponents is \textit{$\sigma$-acceptable} if either
\[1\leq q<\infty,\quad0\leq\theta\leq1,\quad\frac{1}{q}<\sigma\theta\]
or $(q,\theta)=(\infty,0)$.
\end{definition}

If $(\B_0,\B_1)$ is a Banach interpolation couple then write $\B_{\theta,q}$ for the real interpolation space $(\B_0,\B_1)_{\theta,q}$.

\begin{theorem}\label{th:inhomogeneous}
Suppose that $\sigma>0$ and that $\{U(t):t\in\R\}$ satisfies the energy estimate (\ref{est:energy}) and the dispersive estimate (\ref{est:untrunc_decay}). Suppose also that the exponents pairs $(q,\theta)$ and $(\wt{q},\wt{\theta})$ are $\sigma$-acceptable and satisfy the scaling condition
\begin{equation}\label{eq:inhomogeneous scaling condition}
\frac{1}{q}+\frac{1}{\wt{q}}=\frac{\sigma}{2}(\theta+\wt{\theta}).
\end{equation}
\begin{enumerate}
\item[(i)] If
\begin{equation}\label{eq:nonsharp q}
\frac{1}{q}+\frac{1}{\wt{q}}<1,
\end{equation}
\begin{equation}\label{eq:nonsharp theta}
(\sigma-1)(1-\theta)\leq\sigma(1-\wt{\theta}),\qquad(\sigma-1)(1-\wt{\theta})\leq\sigma(1-\theta),
\end{equation}
and, in the case when $\sigma=1$, we have $\theta<1$ and $\wt{\theta}<1$, then the inhomogeneous Strichartz estimate (\ref{eq:str3}) holds.
\item[(ii)] If $q,\wt{q}\in(1,\infty)$
\begin{equation}\label{eq:sharp q}
\frac{1}{q}+\frac{1}{\wt{q}}=1
\end{equation}
and
\begin{equation}\label{eq:sharp theta}
(\sigma-1)(1-\theta)<\sigma(1-\wt{\theta}),\qquad(\sigma-1)(1-\wt{\theta})<\sigma(1-\theta)
\end{equation}
then the inhomogeneous Strichartz estimate
\begin{equation}\label{eq:str4}
\norm{(TT^*)_RF}\stn{q}{(\B_{\theta,q'})^*}\lesssim\norm{F}\stn{\wt{q}'}{\B_{\wt{\theta},\wt{q}'}}
\qquad\forall F\in\st{\wt{q}'}{\B_{\wt{\theta},\wt{q}'}}\cap\st{1}{\B_0}
\end{equation}
holds.
\end{enumerate}
\end{theorem}

\begin{remark}
Suppose that the scaling condition (\ref{eq:inhomogeneous scaling condition}) holds. Then the exponent conditions appearing in (i) and (ii) above are always satisfied if $\sigma<1$ or if $\sigma=1$, $\theta<1$ and $\wt{\theta}<1$.
\end{remark}

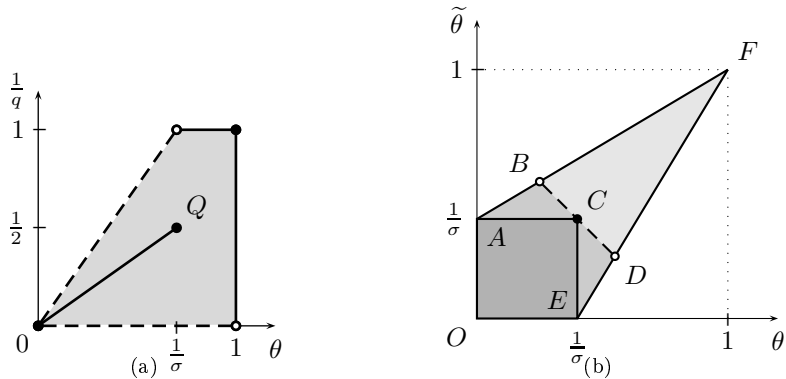
\begin{figure}
\centering
\subfigure[]
{
\begin{pspicture}(-.1,-.2)(3,3)
	\psset{unit=2.6cm}
	\newgray{gray2}{.85}
	\pspolygon[linecolor=white, fillcolor=gray2, fillstyle=solid](0,0)(.7,1)(1,1)(1,0)
	\psline[linewidth=.5pt]{->}(0,0)(0,1.2)
	\psline[linewidth=.5pt]{o->}(1,0)(1.2,0)
	\psline[linewidth=1pt]{*-*}(0,0)(.7,.5)
	\psline[linewidth=1pt, linestyle=dashed]{*-o}(0,0)(1,0)
	\psline[linewidth=1pt]{o-*}(1,0)(1,1)
	\psline[linewidth=1pt, linestyle=dashed]{*-o}(0,0)(.7,1)
	\psline[linewidth=1pt]{o-*}(.7,1)(1,1)
	\psdots[dotstyle=|](.7,0)
	\psdots[dotstyle=|, dotangle=90](0,.5)(0,1)
	\uput[180](0,.5){$\frac{1}{2}$}
	\uput[270](.7,0){$\frac{1}{\sigma}$}
	\uput[dl](0,0){$0$}
	\uput[45](.7,.5){$Q$}
	\uput[l](0,1){$1$}
	\uput[d](1,0){$1$}
	\uput[d](1.2,0){$\theta$}
	\uput[l](0,1.2){$\frac{1}{q}$}
\end{pspicture}
}
\hspace*{2cm}
\subfigure[]
{
\begin{pspicture}(-.3,-.3)(3.6,3.6)
	\psset{unit=3.3cm}
	\newgray{gray1}{.9}
	\newgray{gray2}{.8}
	\newgray{gray3}{.7}
	\psline[linewidth=.4pt]{->}(0,0)(0,1.2)
	\psline[linewidth=.4pt]{->}(0,0)(1.2,0)
	\uput[l](0,1){$1$}
	\uput[d](1,0){$1$}
	\uput[d](1.2,0){$\theta$}
	\uput[l](0,1.2){$\wt{\theta}$}
	\uput[d](.4,0){$\frac{1}{\sigma}$}
	\uput[l](0,.4){$\frac{1}{\sigma}$}
	\pspolygon[linestyle=none, fillcolor=gray2, fillstyle=solid](.4,0)(.4,.4)(.55,.25)
	\pspolygon[linestyle=none, fillcolor=gray2, fillstyle=solid](0,.4)(.25,.55)(.4,.4)
	\pspolygon[linestyle=none, fillcolor=gray1, fillstyle=solid](.55,.25)(.25,.55)(1,1)
	\pspolygon[linewidth=.8pt, fillcolor=gray3, fillstyle=solid](0,0)(.4,0)(.4,.4)(0,.4)
	\psline[linewidth=.8pt](.4,0)(.55,.25)
	\psline[linewidth=.8pt](0,.4)(.25,.55)
	\psline[linewidth=.8pt](.55,.25)(1,1)
	\psline[linewidth=.8pt](1,1)(.25,.55)
	\psline[linewidth=.8pt,linestyle=dashed]{o-o}(.55,.25)(.25,.55)
	\psdot(.4,.4)
	\psline[linewidth=.5pt, linestyle=dotted](1,0)(1,1)(0,1)
	\psdots[dotstyle=|](1,0)
	\psdots[dotstyle=|, dotangle=90](0,1)
	\uput[225](0,0){$O$}
	\uput[315](0,.4){$A$}
	\uput[135](.25,.55){$B$}
	\uput[45](.4,.4){$C$}
	\uput[315](.55,.25){$D$}
	\uput[135](.4,0){$E$}
	\uput[45](1,1){$F$}
\end{pspicture}
}
\caption{Range of exponents for Theorem \ref{th:inhomogeneous} when $\sigma>1$.}
\label{fig:inhomogeneous exponents}
\end{figure}

The closure in the cube $[0,1]^4$ of the set of points $(1/q,\theta,1/\wt{q},\wt{\theta})$  that satisfy the hypotheses of Theorem \ref{th:inhomogeneous} forms a convex polyhedral solid. Projections of this set onto the $\theta q^{-1}$--plane and the $\theta\wt{\theta}$--plane, as well as the improvement of Theorem \ref{th:inhomogeneous} over Theorem \ref{th:KT98}, are shown in Figure \ref{fig:inhomogeneous exponents}. In Figure \ref{fig:inhomogeneous exponents} (a), the closed line segment $OQ$ corresponds to sharp $\sigma$-admissible pairs $(q,\theta)$, while the shaded region represents the set of $\sigma$-acceptable pairs. The region $AOEDB$ in Figure \ref{fig:inhomogeneous exponents} (b) illustrates pairs $(\theta,\wt{\theta})$ that correspond to  inhomogeneous Strichartz estimates given by Theorem \ref{th:inhomogeneous}. In contrast, the region $AOEC$ represents pairs $(\theta,\wt{\theta})$ that correspond to valid exponents for Theorem \ref{th:KT98}.

The proof of Theorem \ref{th:inhomogeneous} is given in Sections \ref{s:preliminaries} to \ref{s:global estimates}, using techniques introduced in \cite{KT98} and adapting part of the argument of \cite{dF05a}. Section \ref{s:preliminaries} states some preliminary results that are used in later sections. Section \ref{s:local estimates} presents inhomogeneous Strichartz estimates that are localised in time. In particular, we prove a local estimate corresponding to the point $F$ in Figure \ref{fig:inhomogeneous exponents} (b), and then interpolate between $F$ and the square $AOEC$ (which corresponds to local estimates implied by Theorem \ref{th:KT98}) to obtain local estimates for exponents in the region $AOEF$. The global inhomogeneous estimates of Theorem \ref{th:inhomogeneous} are then obtained by decomposing the operator $(TT^*)_R$ dyadically as a sum of operators (see Section \ref{s:dyadic}) and estimating each term in the sum by a local inhomogeneous estimate. The summability of the local estimates is obtained in Sections \ref{s:dyadic} and \ref{s:global estimates} by imposing further conditions on the exponents, from which we deduce global inhomogeneous estimates in the region $AOEDB$.

The sharpness of Theorem \ref{th:inhomogeneous} is discussed in Section \ref{s:sharpness}. The rest of the paper is devoted to applications of Theorem \ref{th:inhomogeneous}. In Section \ref{s:schrodinger}, we indicate how the results of Vilela \cite{mV07} and Foschi \cite{dF05a} for the Schr\"odinger equation may be recovered from Theorem \ref{th:inhomogeneous}. It is then shown that, for a certain class of potentials, our generalisation allows one to obtain new Strichartz estimates for Schr\"odinger equations that cannot be deduced from the more specialised theorem of \cite{dF05a}. In Section \ref{s:wave}, new inhomogeneous Strichartz-type estimates are obtained for the wave equation by using homogeneous Besov spaces, in the same spirit as those presented by J. Ginibre and G. Velo \cite{GV95}. Finally, we indicate how one derives Strichartz estimates for the Klein--Gordon equation, thus recovering all the inhomogeneous estimates found by Nakamura and Ozawa \cite{NO01}, as well as giving some new `boundary' estimates.

\section{Some preliminaries}\label{s:preliminaries}

In this section we give some basic tools that will be used in Theorem \ref{th:inhomogeneous}. First we introduce a scaling invariance result that generalises the observation in \cite[Section 1]{KT98}. Second, we present a bilinear formulation of the inhomogeneous Strichartz estimate (\ref{eq:str3}). Finally, we make some remarks on real and complex interpolation of  vector-valued Lebesgue spaces.

In Section \ref{s:local estimates}, a scaling argument will be used to simplify the proof of Theorem \ref{th:inhomogeneous}. The hypotheses (\ref{est:energy}) and (\ref{est:untrunc_decay}) are invariant under many rescalings; we regard the one introduced below to be the simplest.

\begin{proposition}\label{prop:scaling}
If $\lambda>0$ then the estimates (\ref{est:energy}) and (\ref{est:untrunc_decay}) are invariant under the scaling
\begin{equation}\label{scaling}
\begin{cases}
U(t)&\leftarrow U(t/\lambda)\\
\ip<f,g>&\leftarrow\ip<f,g>\\
\norm{f}_{\B_0}&\leftarrow\norm{f}_{\B_0}\\
\norm{f}_{\B_1}&\leftarrow\lambda^{\sigma/2}\norm{f}_{\B_1}.
\end{cases}
\end{equation}
\end{proposition}

The proposition can be easily verified using the following two lemmata.

\begin{lemma}
Suppose that $\B$ is a Banach space and $\lambda>0$. Then the scaling
\begin{equation}\label{eq:dual scaling}
\norm{f}_{\B}\leftarrow\lambda\norm{f}_{\B}\qquad\forall f\in\B
\end{equation}
induces the scaling
\[\norm{\phi}_{\B^*}\leftarrow\lambda^{-1}\norm{\phi}_{\B^*}\qquad\forall \phi\in\B^*.\]
\end{lemma}

The second lemma examines the effect of scaling (\ref{scaling}) on real interpolated spaces. One way of constructing a real interpolation space from an interpolation couple $(\B_0,\B_1)$ is known as \textit{the K-method} \cite[1.3]{hT78}. When $0<t<\infty$ and $b\in \B_0+\B_1$, define $K$ by the formula
\[K(t,b,\B_0,\B_1)=\inf_{b=b_0+b_1}(\norm{b_0}_{\B_0}+t\norm{b_1}_{\B_1}).\]
If $\theta\in(0,1)$ and $q\in[1,\infty)$, then we construct the interpolation space $(\B_0,\B_1)_{\theta,q}$ by
\[(\B_0,\B_1)_{\theta,q}=\left\{b\in \B_0+\B_1:
\norm{b}_{(\B_0,\B_1)_{\theta,q}}<\infty\right\}\]
where
\[\norm{b}_{(\B_0,\B_1)_{\theta,q}}=\left(\int_0^{\infty}\Big(t^{-\theta}K(t,b,\B_0,\B_1)\Big)^q\,\frac{\dd t}{t}\right)^{1/q}.\]

\begin{lemma}
If $\theta\in[0,1]$ then scaling (\ref{scaling}) implies the scaling
\[\norm{f}_{\Bth}\leftarrow\lambda^{\sigma\theta/2}\norm{f}_{\Bth}.\]
\end{lemma}

\begin{proof} 
Whenever $\vartheta\in[0,1]$, let $\B_{\vartheta}'$ denote the Banach space induced from $\B_{\vartheta}$ under scaling (\ref{scaling}). Suppose that $f\in\Bth$. It is clear that
\[K(t,f,\B_0',\B_1')=K(\lambda^{\sigma/2}t,f,\B_0,\B_1)\]
whenever $t>0$. Hence
\begin{align*}
\norm{f}_{\Bth'}
&=\left(\int_0^{\infty}\Big(t^{-\theta}K(t,f,\B_0',\B_1')\Big)^2\,\frac{\dd t}{t}\right)^{1/2}\\
&=\left(\int_0^{\infty}\Big(t^{-\theta}K(\lambda^{\sigma/2}t,f,\B_0,\B_1)\Big)^2\,\frac{\dd t}{t}\right)^{1/2}\\
&=\left(\int_0^{\infty}\Big(\big(\lambda^{-\sigma/2}s\big)^{-\theta}K(s,f,\B_0,\B_1)\Big)^2\,\frac{\dd s}{s}\right)^{1/2}\\
&=\lambda^{\sigma\theta/2}\norm{f}_{\Bth}
\end{align*}
and the lemma follows.
\end{proof}

Following the lead of \cite{KT98}, the Strichartz estimate expressed in (\ref{eq:str3}) as an operator estimate will be re-expressed as a bilinear estimate, thus facilitating flexibility in manipulation and interpolation. When $\B$ is not the Hilbert space $\Hi$, denote by $\ip<f,g>_{\B}$ the action of a linear functional $g$ on an element $f$ of $\B$. Suppose that $F$ and $G$ are in $\st{1}{\B_0}$. Then
\begin{align*}
\ip<(TT^*)_RF,G>\stn{\infty}{\B_0^*}
&=\int_{\R}\ip<U(t)\int_{-\infty}^tU(s)^*F(s)\,\dd s,G(t)>_{\B_0^*}\,\dd t\\
&=\iint_{s<t}\ip<U(s)^*F(s),U(t)^*G(t)>\,\dd s\,\dd t.
\end{align*}
Now define the bilinear form $B$ on $\st{1}{\B_0}\times\st{1}{\B_0}$ by
\begin{equation}\label{eq:T(F,G)}
B(F,G)=\iint_{s<t}\ip<U(s)^*F(s),U(t)^*G(t)>\,\dd s\,\dd t.
\end{equation}
It is not hard to prove the following proposition.

\begin{proposition}\label{prop:retarded estimate equiv bilinear}
Suppose that $q,\wt{q}\in[1,\infty]$ and $\theta,\wt{\theta}\in[0,1]$. Then the inhomogeneous Strichartz estimate (\ref{eq:str3}) is equivalent to the bilinear estimate
\begin{multline}\label{est:bilinear,retarded equiv}
|B(F,G)|\lesssim\norm{F}\stn{\wt{q}'}{\Bwth}\norm{G}\stn{q'}{\Bth}\\
\qquad\forall F\in\st{\wt{q}'}{\Bwth}\cap\st{1}{\B_0} \quad \forall G\in\st{q'}{\Bwth}\cap\st{1}{\B_0},
\end{multline}
where the bilinear form $B$ is given by (\ref{eq:T(F,G)}).
\end{proposition}

Finally, standard results of real and complex interpolation for the spaces $L^q(\R;\Bth)$ will be used several times throughout the next three sections. Readers unfamiliar with the results used are directed to \cite{BL76} and \cite{hT78}. Attention is drawn specifically to \cite[Theorem 4.7.2]{BL76}, \cite[Theorem 5.1.2]{BL76}, \cite[p. 130]{BL76}, \cite[p.76]{BL76} and \cite[p. 128]{hT78}. The negative result of M. Cwikel \cite{mC76} shows the limitations for real interpolation of such spaces.

\section{Local inhomogeneous Strichartz estimates}\label{s:local estimates}

Our proof of Theorem \ref{th:inhomogeneous} spans the next three sections, the first two of which closely follow \cite[Sections 2 and 3]{dF05a}. The main result of this section gives the existence of localised inhomogeneous Strichartz estimates. The following lemma is a preliminary version of this result.

\begin{lemma}\label{lem:local inhomogeneous}
Suppose that $\sigma>0$ and that $\{U(t):t\in\R\}$ satisfies the energy estimate (\ref{est:energy}) and the dispersive estimate (\ref{est:untrunc_decay}). Assume also that $I$ and $J$ are two time intervals of unit length separated by a distance of scale $1$ (that is, $|I|=|J|=1$ and $\dist(I,J)\approx1$). Then the local inhomogeneous Strichartz estimate
\begin{equation}\label{est:local inhomogeneous}
\norm{TT^*F}\stnJ{q}{\Bth^*}\lesssim\norm{F}\stnI{\wt{q}'}{\Bwth}\qquad\forall F\in\stI{\wt{q}'}{\Bwth}\cap\stI{1}{\B_0}
\end{equation}
holds whenever the pairs $(q,\theta)$ and $(\wt{q},\wt{\theta})$ satisfy the conditions
\begin{align}
q,\wt{q}\in[1,\infty]&,\qquad \theta,\wt{\theta}\in[0,1],\label{eq:local1}\\
(\sigma-1)(1-\theta)\leq\sigma(1-\wt{\theta})&,\qquad(\sigma-1)(1-\wt{\theta})\leq\sigma(1-\theta),\label{eq:local2}\\
\frac{1}{q}\geq\frac{\sigma}{2}(\theta-\wt{\theta})&,\qquad\frac{1}{\wt{q}}\geq\frac{\sigma}{2}(\wt{\theta}-\theta)
\label{eq:local3}.
\end{align}
If $\sigma=1$ then $\theta$ and $\wt{\theta}$ must be strictly less than $1$.
\end{lemma}

The lemma and other results appearing in the ensuing sections are proved using a localised version of Proposition \ref{prop:retarded estimate equiv bilinear}. Given two intervals $I$ and $J$ of $\R$, write $I\times J$ as $Q$ and define $B_Q$ by the formula
\begin{equation}\label{eq:B_Q def}
B_Q(F,G)=B(1_IF,1_JG)=\iint_{(s,t)\in I\times J}\ip<U(s)^*F(s),U(t)^*G(t)>\,\dd s\,\dd t
\end{equation}
whenever $F$ and $G$ belong to $L^1(\R;\B_0)$.
One can easily show that the local inhomogeneous estimate (\ref{est:local inhomogeneous}) is equivalent to the bilinear estimate
\begin{multline}\label{est:B_Q}
|B_Q(F,G)|\lesssim\norm{F}\stnI{\wt{q}'}{\Bwth}\norm{G}\stnJ{q'}{\Bth}\\
\forall F\in\stI{\wt{q}'}{\Bwth}\cap\stI{1}{\B_0}\quad\forall G\in\stJ{q'}{\Bth}\cap\stJ{1}{\B_0}.
\end{multline}

\begin{proof}[Proof of Lemma \ref{lem:local inhomogeneous}]
Suppose that $I$ and $J$ are two intervals satisfying the hypothesis of the theorem and write $I\times J$ as $Q$.
Let $\Psi$ denote the set of points $(1/q,\theta;1/\wt{q},\wt{\theta})$ in $[0,1]^4$ corresponding to the pairs $(q,\theta)$ and $(\wt{q},\wt{\theta})$ for which estimate (\ref{est:local inhomogeneous}), or its bilinear equivalent (\ref{est:B_Q}), is valid.

The dispersive estimate (\ref{est:untrunc_decay}) implies that
\begin{align}
|B_Q(F,G)|
&\leq \iint_Q|\ip<F(s),U(s)U(t)^*G(t)>_{\B_1}|\,\dd s\,\dd t\notag\\
&\leq\iint_Q \norm{F(s)}_{\B_1}\norm{U(s)U(t)^*G(t)}_{B_1^*}\,\dd s\,\dd t\notag\\
&\lesssim \int_J\int_I|t-s|^{-\sigma}\norm{F(s)}_{\B_1}\norm{G(t)}_{\B_1}\,\dd s\,\dd t\notag\\
&\lesssim\norm{F}\stnI{1}{\B_1}\norm{G}\stnJ{1}{\B_1}.\label{est:local (0,1;0,1)}
\end{align}
Hence $(0,1;0,1)\in\Psi$.
On the other hand, the dual
\[\norm{\int_{\R}U(s)^*F(s)\,\dd s}_{\Hi}\lesssim\norm{F}\stn{q'}{\Bth}
\qquad\forall F\in\st{q'}{\Bth}\cap\st{1}{\B_0}\]
of the homogeneous Strichartz estimate (\ref{eq:str1}) of Theorem \ref{th:KT98} implies that
\begin{align}\label{est:local sharp sigma admiss}
|B_Q(F,G)|
&\leq\norm{\int_IU(s)^*F(s)\,\dd s}_{\Hi}\norm{\int_JU(t)^*G(t)\,\dd s}_{\Hi}\notag\\
&\lesssim\norm{F}\stnI{\wt{q}'}{\Bwth}\norm{G}\stnJ{q'}{\Bth}
\end{align}
whenever $(q,\theta)$ and $(\wt{q},\wt{\theta})$ are sharp $\sigma$-admissible.
Complex interpolation between (\ref{est:local (0,1;0,1)}) and (\ref{est:local sharp sigma admiss}) shows that $\Psi$ contains the convex hull of the set
\begin{equation}\label{set:convex hull}
(0,1;0,1)\cup\left\{(1/q,\theta;1/\wt{q},\wt{\theta}):
(q,\theta)\mbox{ and }(\wt{q},\wt{\theta})\mbox{ are $\sigma$-admissible pairs}\right\}.
\end{equation}

Since $G$ is restricted to a unit time interval, H\"older's inequality gives
\[
\norm{G}\stnJ{q'}{\Bth}
=\norm{1_JG}\stnJ{q'}{\Bth}
\leq\norm{1_J}_{L^{r'}(J)}\norm{G}\stnJ{p'}{\Bth}
\lesssim\norm{G}\stnJ{p'}{\Bth}
\]
whenever $1/q'=1/r'+1/p'$. We can always perform this calculation provided that $p\leq q$. Similarly, if $\wt{p}\leq \wt{q}$ then
\[\norm{F}\stnI{\wt{q}'}{\Bwth}\lesssim\norm{F}\stnI{\wt{p}'}{\Bwth}.\]
Hence if $(1/q,\theta;1/\wt{q},\wt{\theta})\in\Psi$ then $(1/p,\theta;1/\wt{p},\wt{\theta})\in\Psi$ whenever $p\leq q$ and $\wt{p}\leq\wt{q}$. If we apply this property to the points of the convex hull of (\ref{set:convex hull}) then we obtain a set $\Psi_*$, contained in $\Psi$, that is described precisely by the conditions appearing in Lemma \ref{lem:local inhomogeneous}. Details of this computation are analogous to those of \cite[Appendix A]{dF05a} and will be omitted.
\end{proof}

Recall (see Proposition \ref{prop:scaling}) that the energy estimate (\ref{est:energy}) and dispersive estimate (\ref{est:untrunc_decay}) are invariant with respect to the rescaling (\ref{scaling}). We shall apply this scaling to the local inhomogeneous estimate (\ref{est:local inhomogeneous}) to obtain a version of Lemma \ref{lem:local inhomogeneous} for intervals $I$ and $J$ that are not of unit length.

When scaling (\ref{scaling}) is applied, (\ref{est:local inhomogeneous}) becomes
\[\lambda^{-\sigma\theta/2}\left(\int_J\norm{\int_{\R}U(t/\lambda)U(s/\lambda)^*F(s)\,\dd s}_{\Bth^*}^q\,\dd t\right)^{1/q}\leq C\lambda^{\sigma\wt{\theta}/2}\norm{F}\stnI{\wt{q}'}{\Bwth}.\]
The substitution $s\mapsto\lambda s$ gives
\[\lambda^{-\sigma\theta/2+1}\left(\int_J\norm{(TT^*F_0)(t/\lambda)}_{\Bth^*}^q\,\dd t\right)^{1/q}\leq C\lambda^{\sigma\wt{\theta}/2+1/\wt{q}'}\norm{F_0}_{L^{\wt{q}'}(\lambda^{-1}I;\Bwth)}\]
where $F_0(s)=F(\lambda s)$. A further substitution $t\mapsto\lambda t$ yields
\[\lambda^{-\sigma\theta/2+1+1/q}\norm{TT^*F_0}_{L^q(\lambda^{-1}J;\Bth^*)}\leq C\lambda^{\sigma\wt{\theta}/2+1-1/\wt{q}}\norm{F_0}_{L^{\wt{q}'}(\lambda^{-1}I;\Bwth)}.\]
Hence
\[\norm{TT^*F_0}_{L^q(\lambda^{-1}J;\Bth^*)}\leq C\lambda^{-\beta(q,\theta;\wt{q},\wt{\theta})}\norm{F_0}_{L^{\wt{q}'}(\lambda^{-1}I;\Bwth)}\]
where
\begin{equation}\label{eq:4 parameter beta}
\beta(q,\theta;\wt{q},\wt{\theta})=\frac{1}{q}+\frac{1}{\wt{q}}-\frac{\sigma}{2}(\theta+\wt{\theta}).
\end{equation}
If we replace $\lambda$ with $\lambda^{-1}$ in the last inequality then we obtain the following Theorem.

\begin{theorem}\label{th:scaled local inhomogeneous}
Suppose that $\sigma>0$, $\lambda>0$ and $\{U(t):t\in\R\}$ satisfies the energy estimate (\ref{est:energy}) and the untruncated decay estimate (\ref{est:untrunc_decay}). Assume also that $I$ and $J$ are two time intervals of length $\lambda$ separated by a distance of scale $\lambda$ (that is, $|I|=|J|=\lambda$ and $\dist(I,J)\approx\lambda$). Then the local inhomogeneous Strichartz estimate
\begin{equation}\label{est:scaled local inhomogeneous}
\norm{TT^*F}\stnJ{q}{\Bth^*}\lesssim\lambda^{\beta(q,\theta;\wt{q},\wt{\theta})}\norm{F}\stnI{\wt{q}'}{\Bwth}\qquad\forall F\in\stI{\wt{q}'}{\Bwth}\cap\stI{1}{\B_0}
\end{equation}
holds whenever the pairs $(q,\theta)$ and $(\wt{q},\wt{\theta})$ satisfy the conditions appearing in Lemma \ref{lem:local inhomogeneous}.
\end{theorem}

\section{Dyadic decompositions}\label{s:dyadic}

Theorem \ref{th:scaled local inhomogeneous} gives spacetime estimates for $(TT^*)_R$ which are localised in time. In this section we make first steps toward obtaining global spacetime estimates for this operator.

We begin with a few preliminaries. We say that $\lambda$ is a dyadic number if $\lambda=2^k$ for some integer $k$. Denote by $2^{\Z}$ the set of all dyadic numbers. We say that a square in $\R^2$ is a \textit{dyadic square} if its side length $\lambda$ is a dyadic number and if the all the coordinates if its vertices are integer multiples of $\lambda$. It is well known (see, for example, \cite[Appendix J]{lG04}) that any open set $\Omega$ in $\R^2$ can be decomposed as the union of essentially disjoint dyadic squares whose lengths are approximately proportional to their distance from the boundary $\partial\Omega$ of $\Omega$. Such a decomposition is known as a \textit{Dyadic Whitney decomposition} of $\Omega$.

\begin{figure}

\centering

\begin{pspicture}(-.5,-.5)(7,7)
	\psset{unit=.1cm}
	\pspolygon[linecolor=white, fillcolor=lightgray, fillstyle=solid](0,0)(64,0)(64,64)
	\psline[linestyle=dotted](0,0)(64,64)
	\psline[linewidth=.4pt](-1,64)(63,64)
	\psline[linewidth=.4pt](-1,32)(31,32)
	\psline[linewidth=.4pt](0,48)(47,48)
	\psline[linewidth=.4pt](0,48)(47,48)
	\psline[linewidth=.4pt](-1,16)(15,16)
	\psline[linewidth=.4pt](0,65)(0,1)
	\psline[linewidth=.4pt](16,64)(16,17)
	\psline[linewidth=.4pt](32,65)(32,33)
	\psline[linewidth=.4pt](48,65)(48,49)
	\psline[linewidth=.4pt](-1,8)(7,8)
	\psline[linewidth=.4pt](0,24)(23,24)
	\psline[linewidth=.4pt](16,40)(39,40)
	\psline[linewidth=.4pt](32,56)(55,56)
	\psline[linewidth=.4pt](8,9)(8,32)
	\psline[linewidth=.4pt](24,25)(24,48)
	\psline[linewidth=.4pt](40,41)(40,64)
	\psline[linewidth=.4pt](56,57)(56,65)
	\psline[linewidth=.4pt](-1,4)(3,4)
	\psline[linewidth=.4pt](0,12)(11,12)
	\psline[linewidth=.4pt](8,20)(19,20)
	\psline[linewidth=.4pt](16,28)(27,28)
	\psline[linewidth=.4pt](24,36)(35,36)
	\psline[linewidth=.4pt](32,44)(43,44)
	\psline[linewidth=.4pt](40,52)(51,52)
	\psline[linewidth=.4pt](48,60)(59,60)
	\psline[linewidth=.4pt](4,5)(4,16)
	\psline[linewidth=.4pt](12,13)(12,24)
	\psline[linewidth=.4pt](20,21)(20,32)
	\psline[linewidth=.4pt](28,29)(28,40)
	\psline[linewidth=.4pt](36,37)(36,48)
	\psline[linewidth=.4pt](44,45)(44,56)
	\psline[linewidth=.4pt](52,53)(52,64)
	\psline[linewidth=.4pt](60,61)(60,65)
	\psline[linewidth=.4pt](-1,2)(1,2)
	\psline[linewidth=.4pt](0,6)(5,6)
	\psline[linewidth=.4pt](4,10)(9,10)
	\psline[linewidth=.4pt](8,14)(13,14)
	\psline[linewidth=.4pt](12,18)(17,18)
	\psline[linewidth=.4pt](16,22)(21,22)
	\psline[linewidth=.4pt](20,26)(25,26)
	\psline[linewidth=.4pt](24,30)(29,30)
	\psline[linewidth=.4pt](28,34)(33,34)
	\psline[linewidth=.4pt](32,38)(37,38)
	\psline[linewidth=.4pt](36,42)(41,42)
	\psline[linewidth=.4pt](40,46)(45,46)
	\psline[linewidth=.4pt](44,50)(49,50)
	\psline[linewidth=.4pt](48,54)(53,54)
	\psline[linewidth=.4pt](52,58)(57,58)
	\psline[linewidth=.4pt](56,62)(61,62)
	\psline[linewidth=.4pt](2,3)(2,8)
	\psline[linewidth=.4pt](6,7)(6,12)
	\psline[linewidth=.4pt](10,11)(10,16)
	\psline[linewidth=.4pt](14,15)(14,20)
	\psline[linewidth=.4pt](18,19)(18,24)
	\psline[linewidth=.4pt](22,23)(22,28)
	\psline[linewidth=.4pt](26,27)(26,32)
	\psline[linewidth=.4pt](30,31)(30,36)
	\psline[linewidth=.4pt](34,35)(34,40)
	\psline[linewidth=.4pt](38,39)(38,44)
	\psline[linewidth=.4pt](42,43)(42,48)
	\psline[linewidth=.4pt](46,47)(46,52)
	\psline[linewidth=.4pt](50,51)(50,56)
	\psline[linewidth=.4pt](54,55)(54,60)
	\psline[linewidth=.4pt](58,59)(58,64)
	\psline[linewidth=.4pt](62,63)(62,65)
	\psline[linewidth=.4pt](0,3)(2,3)
	\psline[linewidth=.4pt](2,5)(4,5)
	\psline[linewidth=.4pt](4,7)(6,7)
	\psline[linewidth=.4pt](6,9)(8,9)
	\psline[linewidth=.4pt](8,11)(10,11)
	\psline[linewidth=.4pt](10,13)(12,13)
	\psline[linewidth=.4pt](12,15)(14,15)
	\psline[linewidth=.4pt](14,17)(16,17)
	\psline[linewidth=.4pt](16,19)(18,19)
	\psline[linewidth=.4pt](18,21)(20,21)
	\psline[linewidth=.4pt](20,23)(22,23)
	\psline[linewidth=.4pt](22,25)(24,25)
	\psline[linewidth=.4pt](24,27)(26,27)
	\psline[linewidth=.4pt](26,29)(28,29)
	\psline[linewidth=.4pt](28,31)(30,31)
	\psline[linewidth=.4pt](30,33)(32,33)
	\psline[linewidth=.4pt](32,35)(34,35)
	\psline[linewidth=.4pt](34,37)(36,37)
	\psline[linewidth=.4pt](36,39)(38,39)
	\psline[linewidth=.4pt](38,41)(40,41)
	\psline[linewidth=.4pt](40,43)(42,43)
	\psline[linewidth=.4pt](42,45)(44,45)
	\psline[linewidth=.4pt](44,47)(46,47)
	\psline[linewidth=.4pt](46,49)(48,49)
	\psline[linewidth=.4pt](48,51)(50,51)
	\psline[linewidth=.4pt](50,53)(52,53)
	\psline[linewidth=.4pt](52,55)(54,55)
	\psline[linewidth=.4pt](54,57)(56,57)
	\psline[linewidth=.4pt](56,59)(58,59)
	\psline[linewidth=.4pt](58,61)(60,61)
	\psline[linewidth=.4pt](60,63)(62,63)
	\psline[linewidth=.4pt](1,2)(1,4)
	\psline[linewidth=.4pt](3,4)(3,6)
	\psline[linewidth=.4pt](5,6)(5,8)
	\psline[linewidth=.4pt](7,8)(7,10)
	\psline[linewidth=.4pt](9,10)(9,12)
	\psline[linewidth=.4pt](11,12)(11,14)
	\psline[linewidth=.4pt](13,14)(13,16)
	\psline[linewidth=.4pt](15,16)(15,18)
	\psline[linewidth=.4pt](17,18)(17,20)
	\psline[linewidth=.4pt](19,20)(19,22)
	\psline[linewidth=.4pt](21,22)(21,24)
	\psline[linewidth=.4pt](23,24)(23,26)
	\psline[linewidth=.4pt](25,26)(25,28)
	\psline[linewidth=.4pt](27,28)(27,30)
	\psline[linewidth=.4pt](29,30)(29,32)
	\psline[linewidth=.4pt](31,32)(31,34)
	\psline[linewidth=.4pt](33,34)(33,36)
	\psline[linewidth=.4pt](35,36)(35,38)
	\psline[linewidth=.4pt](37,38)(37,40)
	\psline[linewidth=.4pt](39,40)(39,42)
	\psline[linewidth=.4pt](41,42)(41,44)
	\psline[linewidth=.4pt](43,44)(43,46)
	\psline[linewidth=.4pt](45,46)(45,48)
	\psline[linewidth=.4pt](47,48)(47,50)
	\psline[linewidth=.4pt](49,50)(49,52)
	\psline[linewidth=.4pt](51,52)(51,54)
	\psline[linewidth=.4pt](53,54)(53,56)
	\psline[linewidth=.4pt](55,56)(55,58)
	\psline[linewidth=.4pt](57,58)(57,60)
	\psline[linewidth=.4pt](59,60)(59,62)
	\psline[linewidth=.4pt](61,62)(61,64)

	\uput[0](16.8,35.6){$Q$}
	\psline[linewidth=1pt](16,32)(16,40)
	\psline[linewidth=1pt](24,40)(16,40)
	\psline[linewidth=1pt](24,40)(24,32)
	\psline[linewidth=1pt](16,32)(24,32)
\end{pspicture}

\caption{Whitney's decomposition for the region $s<t$.}

\label{fig:whitney}
\end{figure}
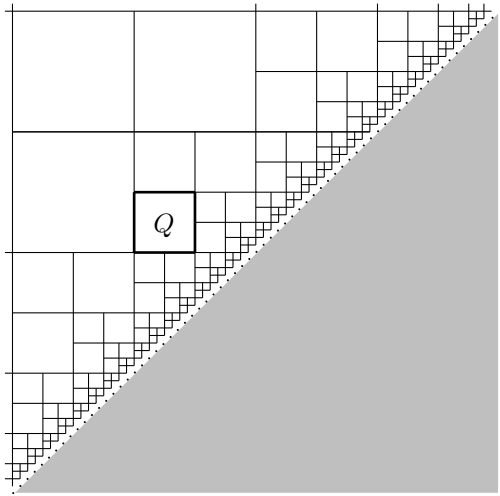

From now on, let $\Q$ denote the Dyadic Whitney decomposition illustrated in Figure \ref{fig:whitney} for the domain $\Omega$, where
\[\Omega=\{(s,t)\in\R^2:s<t\}.\]
For each dyadic number $\lambda$, let $\Q_{\lambda}$ denote the family contained in $\Q$ consisting of squares with side length $\lambda$. Each square $Q$ in $\Q_{\lambda}$ is the Cartesian product $I\times J$ of two intervals of $\R$ and has the property that
\[\lambda=|I|=|J|\approx\dist(Q,\partial\Omega)\approx\dist(I,J).\]
Since the squares $Q$ in the decomposition of $\Omega$ are pairwise essentially disjoint, we have the decomposition
\begin{equation}\label{eq:whitney decomposition of B}
B=\sum_{\lambda\in 2^{\Z}}\sum_{Q\in\Q_{\lambda}}B_Q,
\end{equation}
where $B$ is given by (\ref{eq:T(F,G)}) and $B_Q$ is given by (\ref{eq:B_Q def}) whenever $Q=I\times J$. The scaled version
\begin{multline}\label{est:scaled B_Q}
|B_Q(F,G)|\lesssim\lambda^{\beta(q,\theta;\wt{q},\wt{\theta})}\norm{F}\stnI{\wt{q}'}{\Bwth}\norm{G}\stnJ{q'}{\Bth}\\
\forall F\in\stI{\wt{q}'}{\Bwth}\cap\stI{1}{\B_0}\quad\forall G\in\stJ{q'}{\Bth}\cap\stJ{1}{\B_0}
\end{multline}
of (\ref{est:B_Q}) is equivalent to the scaled local inhomogeneous Strichartz estimate (\ref{est:scaled local inhomogeneous}).
The next proposition will enable us to replace the localised spaces $\stI{\wt{q}'}{\Bwth}$ and $\stJ{q'}{\Bth}$ with $\st{\wt{q}'}{\Bwth}$ and $\st{q'}{\Bth}$ at the cost of imposing another condition on $\wt{q}$ and $q$.

\begin{proposition}\label{prop:sum B_Q}
Suppose that $\sigma>0$, $1/q+1/\wt{q}\leq1$, $\lambda$ is a dyadic number and $\{U(t):t\in\R\}$ satisfies the energy estimate (\ref{est:energy}) and untruncated decay (\ref{est:untrunc_decay}). If the pairs $(q,\theta)$ and $(\wt{q},\wt{\theta})$ satisfy the conditions appearing in Lemma \ref{lem:local inhomogeneous} then
\begin{multline}\label{est:sum B_Q}
\sum_{Q\in\Q_{\lambda}}|B_Q(F,G)|\lesssim\lambda^{\beta(q,\theta;\wt{q},\wt{\theta})}
\norm{F}\stn{\wt{q}'}{\Bwth}\norm{G}\stn{q'}{\Bth}\\
\forall F\in\st{\wt{q}'}{\Bwth}\cap\st{1}{\B_0}\quad\forall G\in\st{q'}{\Bth}\cap\st{1}{\B_0}.
\end{multline}
\end{proposition}

The proposition is an immediate consequence of Theorem \ref{th:scaled local inhomogeneous}, the equivalence of (\ref{est:scaled local inhomogeneous}) and (\ref{est:scaled B_Q}), and \cite[Lemma 3.2]{dF05a}.

\section{Proof of the global inhomogeneous Strichartz estimates}\label{s:global estimates}

To obtain the required global bilinear estimate (\ref{est:bilinear,retarded equiv}), one cannot simply sum (\ref{est:sum B_Q}) over all dyadic numbers $\lambda$, since the right-hand side is not summable in $\lambda$. To overcome this problem, we perturb the exponents slightly and interpolate to gain summability. It is at this stage that we depart from the approach of Foschi \cite[Sections 4 and 5]{dF05a}, whose chief technical tool is $p$-atomic decomposition of $L^p$ functions. In our abstract setting, the luxury of such decompositions for elements of the Banach space $\Bth$ is not present. Instead we prefer to use an abstract argument that appeals to real interpolation theory in much the same way as \cite[Section 6]{KT98}. The advantage of this approach is twofold. First, the proofs are shorter than Foschi's proofs. Second, it admits function spaces other than the Lebesgue spaces.

We require two facts about real interpolation. The first concerns real interpolation of weighted Lebesgue sequence spaces. Whenever $s\in\R$ and $1<q<\infty$, let $\ell_s^q$ denote the space of all scalar-valued sequences $\{a_j\}_{j\in\mathbb{Z}}$ such that
\[\norm{\{a_j\}_{j\in\Z}}_{\ell_s^q}=\Big(\sum_{j\in\mathbb{Z}}2^{js}|a_j|^q\Big)^{1/q}<\infty.\]
If $q=\infty$ then the norm is defined by
\[\norm{\{a_j\}_{j\in\Z}}_{\ell_s^{\infty}}=\sup_{j\in\mathbb{Z}}\,2^{js}|a_j|.\]
A special case of \cite[Theorem 5.6.1]{BL76} says that if $s_0$ and $s_1$ are two different real numbers and $0<\theta<1$ then
\begin{equation}\label{id:lsq_interpolation}
(\ell^{\infty}_{s_0},\ell_{s_1}^{\infty})_{\theta,1}=\ell_s^1,
\end{equation}
where $s=(1-\theta)s_0+\theta s_1$.

The second fact needed is given by the following lemma.

\begin{lemma}\label{lem:real bilinear interpolation}
\cite[pp.~76--77]{BL76}
Suppose that $(\A_0,\A_1)$, $(\B_0,\B_1)$ and $(\C_0,\C_1)$ are interpolation couples and that the bilinear operator $S$ acts as a bounded transformation as indicated below:
\begin{align*}
&S:\A_0\times \B_0\to \C_0 \\
&S:\A_0\times \B_1\to \C_1 \\
&S:\A_1\times \B_0\to \C_1.
\end{align*}
If $\theta_0,\theta_1\in(0,1)$ and $p,q,r\in[1,\infty]$ such that $1\leq1/p+1/q$ and $\theta_0+\theta_1<1$, then $S$ also acts as a bounded transformation in the following way:
\[S:(\A_0,\A_1)_{\theta_0,pr}\times (\B_0,\B_1)_{\theta_1,qr}\to(\C_0,\C_1)_{\theta_0+\theta_1,r}.\]
\end{lemma}

We are now ready to prove the global inhomogeneous Strichartz estimates of Theorem \ref{th:inhomogeneous}.

\begin{lemma}\label{lem:global inhomogeneous perturb q} Suppose that $\sigma>0$ and that $\{U(t):t\in\R\}$ satisfies the energy estimate (\ref{est:energy}) and the dispersive estimate (\ref{est:untrunc_decay}). Then the inhomogeneous Strichartz estimate (\ref{eq:str3}) holds whenever the exponent pairs $(q,\theta)$ and $(\wt{q},\wt{\theta})$ satisfy the conditions
\begin{align*}
q,\wt{q}\in(1,\infty)&,\qquad\theta,\wt{\theta}\in[0,1],\\
(\sigma-1)(1-\theta)\leq\sigma(1-\wt{\theta})&,\qquad(\sigma-1)(1-\wt{\theta})\leq\sigma(1-\theta),\\
\frac{1}{q}>\frac{\sigma}{2}(\theta-\wt{\theta})&,\qquad\frac{1}{\wt{q}}>\frac{\sigma}{2}(\wt{\theta}-\theta),\\
\frac{1}{q}&+\frac{1}{\wt{q}}<1
\end{align*}
and
\begin{equation}\label{eq:inhomogeneous scaling 2}
\frac{1}{q}+\frac{1}{\wt{q}}=\frac{\sigma}{2}(\theta+\wt{\theta}).
\end{equation}
If $\sigma=1$ then we also require that $\theta<1$ and $\wt{\theta}<1$.
\end{lemma}

\begin{proof}
Suppose that the exponent pairs $(q,\theta)$ and $(\wt{q},\wt{\theta})$ satisfy the conditions appearing in the statement of the theorem. Then there is a positive $\epsilon$ such that the pairs $(q_0,\theta)$ and $(\wt{q}_0,\wt{\theta})$ and the pairs $(q_1,\theta)$ and $(\wt{q}_1,\wt{\theta})$, defined by
\[
\frac{1}{q_0}=\frac{1}{q}-\epsilon,\qquad\frac{1}{\wt{q}_0}=\frac{1}{\wt{q}}-\epsilon,\qquad
\frac{1}{q_1}=\frac{1}{q}+2\epsilon,\qquad\frac{1}{\wt{q}_1}=\frac{1}{\wt{q}}+2\epsilon,
\]
also satisfy all the conditions appearing in the statement of the theorem except for (\ref{eq:inhomogeneous scaling 2}).

Define a function $\wt{B}$ on $\st{1}{\B_0}\times\st{1}{\B_0}$ by
\[\wt{B}(F,G)=\left\{\sum_{Q\in\Q_{2^{-j}}} B_Q(F,G)\right\}_{j\in\Z}.\]
Proposition \ref{prop:sum B_Q} implies that the maps
\begin{align*}
\wt{B}:\st{\wt{q}'_0}{\B_{\wt{\theta}}}\times\st{q'_0}{\B_{\theta}}
&\to\ell^{\infty}_{\beta(q_0,\theta;\wt{q}_0,\wt{\theta})}\\
\wt{B}:\st{\wt{q}'_0}{\B_{\wt{\theta}}}\times\st{q'_1}{\B_{\theta}}
&\to\ell^{\infty}_{\beta(q_1,\theta;\wt{q}_0,\wt{\theta})}\\
\wt{B}:\st{\wt{q}'_1}{\B_{\wt{\theta}}}\times\st{q'_0}{\B_{\theta}}
&\to\ell^{\infty}_{\beta(q_0,\theta;\wt{q}_1,\wt{\theta})}
\end{align*}
are bounded. Note that $\beta(q_1,\theta;\wt{q}_0,\wt{\theta})=\beta(q_0,\theta;\wt{q}_1,\wt{\theta})$. So we may apply Lemma \ref{lem:real bilinear interpolation} to obtain the bounded map
\begin{multline}\label{wtB:complicated q interpolation}
\wt{B}:\big(\st{\wt{q}'_0}{\B_{\wt{\theta}}},\st{\wt{q}'_1}{\B_{\wt{\theta}}}\big)_{\eta_0,\wt{q}'}
\times\big(\st{q'_0}{\B_{\theta}},\st{q'_1}{\B_{\theta}}\big)_{\eta_1,q'}\\
\to\big(\ell^{\infty}_{\beta(q_0,\theta;\wt{q}_0,\wt{\theta})},
\ell^{\infty}_{\beta(q_1,\theta;\wt{q}_0,\wt{\theta})}\big)_{\eta,1}
\end{multline}
where $\eta_0=\eta_1=\tfrac{1}{3}$ and $\eta=\eta_0+\eta_1$. It is easy to check that \[(1-\eta)\beta(q_0,\theta;\wt{q}_0,\wt{\theta})+\eta\beta(q_1,\theta;\wt{q}_0,\wt{\theta})
=\beta(q,\theta;\wt{q},\wt{\theta})=0.\]
If we combine this with (\ref{id:lsq_interpolation}) then (\ref{wtB:complicated q interpolation}) simplifies to
\[\wt{B}:\st{\wt{q}'}{\Bwth}\times\st{q'}{\Bth}\to \ell_0^1,\]
from which we obtain the bilinear estimate (\ref{est:bilinear,retarded equiv}).
\end{proof}

In the above proof we first perturbed the time exponents $q$ and $\wt{q}$ in estimate (\ref{est:sum B_Q}) and then interpolated. Successful perturbation required strict inequalities in the conditions appearing in Lemma (\ref{lem:local inhomogeneous}) that involved $q$ and $\wt{q}$. The proof (which we omit) of the next lemma uses the same idea, except that the spatial exponents $\theta$ and $\wt{\theta}$ are perturbed instead. This allows us to recover some boundary cases that the previous lemma excludes. 

\begin{lemma}\label{lem:global inhomogeneous perturb theta} Suppose that $\sigma>0$ and that $\{U(t):t\in\R\}$ satisfies the energy estimate (\ref{est:energy}) and the untruncated decay estimate (\ref{est:untrunc_decay}). Then the inhomogeneous Strichartz estimate (\ref{eq:str4}) holds whenever the exponent pairs $(q,\theta)$ and $(\wt{q},\wt{\theta})$ satisfy the conditions
\begin{align}
q,\wt{q}\in(1,\infty]&,\qquad\theta,\wt{\theta}\in(0,1),\notag\\
(\sigma-1)(1-\theta)<\sigma(1-\wt{\theta})&,\qquad(\sigma-1)(1-\wt{\theta})<\sigma(1-\theta),\notag\\
\frac{1}{q}>\frac{\sigma}{2}(\theta-\wt{\theta})&,\qquad\frac{1}{\wt{q}}>\frac{\sigma}{2}(\wt{\theta}-\theta),
\label{eq:perturb theta 2}\\
\frac{1}{q}&+\frac{1}{\wt{q}}\leq1 \notag
\end{align}
and
\begin{equation*}
\frac{1}{q}+\frac{1}{\wt{q}}=\frac{\sigma}{2}(\theta+\wt{\theta}).
\end{equation*}
\end{lemma}

The two previous lemmata combine to give Theorem \ref{th:inhomogeneous}. For example, suppose that $(q,\theta)$ and $(\wt{q},\wt{\theta})$ satisfy the conditions appearing in Theorem \ref{th:inhomogeneous} case (ii). If $\theta>0$ and $\wt{\theta}>0$ then  $\sigma$-acceptability is equivalent to (\ref{eq:perturb theta 2}) by the scaling condition (\ref{eq:inhomogeneous scaling condition}). In this case, Lemma \ref{lem:global inhomogeneous perturb theta} shows that the retarded Strichartz estimate (\ref{eq:str4}) holds. On the other hand, if either $\theta=0$ or $\wt{\theta}=0$ then $\sigma$-acceptability, (\ref{eq:inhomogeneous scaling condition}) and (\ref{eq:sharp theta}) imply that both $(q,\theta)$ and $(\wt{q},\wt{\theta})$ are sharp $\sigma$-admissible. Hence the Strichartz estimate (\ref{eq:str3}) holds by Theorem \ref{th:KT98}. But since $q\geq2$ and $\wt{q}\geq2$, \cite[Theorem 3.4.1]{BL76} gives the continuous embeddings $\B_{\theta,q'}\subseteq\Bth$ and $\B_{\wt{\theta},\wt{q}'}\subseteq\Bwth$ and thus (\ref{eq:str3}) implies (\ref{eq:str4}).

\section{The sharpness of the main theorem}\label{s:sharpness}

In this section we discuss the sharpness of the exponent conditions appearing in Theorem \ref{th:inhomogeneous}.

\begin{proposition}
Suppose that $\sigma>0$ and that the inhomogeneous Strichartz estimate (\ref{eq:str3}) holds for any $\{U(t):t\geq0\}$ satisfying the energy estimate (\ref{est:energy}) and the dispersive estimate (\ref{est:untrunc_decay}). Then $(q,\theta)$ and $(\wt{q},\wt{\theta})$ must be $\sigma$-admissible pairs that satisfy the following conditions:
\begin{align}
\frac{1}{q}+\frac{1}{\wt{q}}&=\frac{\sigma}{2}(\theta+\wt{\theta}),\label{eq:nec1}\\
\frac{1}{q}+\frac{1}{\wt{q}}&\leq1,\label{eq:nec2}\\
|\theta-\wt{\theta}|&\leq\frac{1}{\sigma}\label{eq:nec3}
\end{align}
and
\begin{equation}\label{eq:nec4}
(\sigma-1)(1-\theta)-\frac{2}{q}\leq\sigma(1-\wt{\theta}),\qquad(\sigma-1)(1-\wt{\theta})-\frac{2}{q}\leq\sigma(1-\theta).
\end{equation}
Moreover, if $\sigma=1$ then the inhomogeneous estimate is false when $\theta=\wt{\theta}=1$.
\end{proposition}

The method of proof is quite standard. Before presenting it, we note that the difference in the necessary and sufficient conditions for the validity of the inhomogeneous Strichartz estimate (\ref{eq:str3}) essentially lies in two places. First there is the gap between (\ref{eq:nec4}) and (\ref{eq:nonsharp theta}). Second, there is the gap between the range of values for $\theta$ and $\wt{\theta}$ as shown in Figure  \ref{fig:sharpness of theta against theta}. In particular, the region $AOEDB$ in corresponds to sufficient conditions for $\theta$ and $\wt{\theta}$ while the region $AOED'B'$ corresponds to necessary conditions. The boundaries of each region are included except the line segment $BD$ for the sufficient conditions. This discrepancy along $BD$ is muted somewhat by the validity of the inhomogeneous estimate (\ref{eq:str4}) when $\frac{\sigma}{2}(\theta+\wt{\theta})=1$.

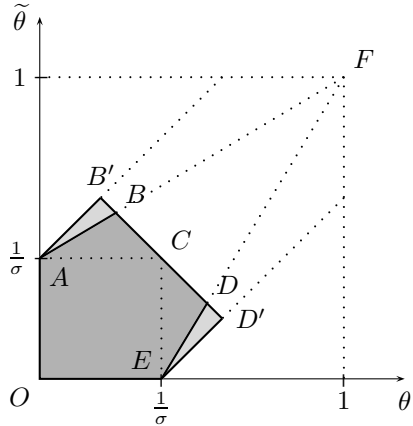
\begin{figure}
\centering
\subfigure
{
\begin{pspicture}(-.4,-.4)(4.4,4.4)
	\psset{unit=4cm}
	\newgray{gray1}{.85}
	\newgray{gray3}{.7}
	\psline[linewidth=.4pt]{->}(0,0)(0,1.2)
	\psline[linewidth=.4pt]{->}(0,0)(1.2,0)
	\uput[l](0,1){$1$}
	\uput[d](1,0){$1$}
	\uput[d](1.2,0){$\theta$}
	\uput[l](0,1.2){$\wt{\theta}$}
	\uput[d](.4,0){$\frac{1}{\sigma}$}
	\uput[l](0,.4){$\frac{1}{\sigma}$}
	\pspolygon[linewidth=.8pt, linecolor=white, fillcolor=gray3, fillstyle=solid](0,0)(0,.4)(.25,.55)(.55,.25)(.4,0)
	\pspolygon[linewidth=.8pt, linecolor=white, fillcolor=gray1, fillstyle=solid](.4,0)(.55,.25)(.6,.2)
	\pspolygon[linewidth=.8pt, linecolor=white, fillcolor=gray1, fillstyle=solid](0,.4)(.25,.55)(.2,.6)
	\psline[linewidth=.8pt](.4,0)(.55,.25)
	\psline[linewidth=.8pt](0,.4)(.25,.55)
	\psline[linewidth=.8pt](.55,.25)(.25,.55)
	\psline[linewidth=.8pt](0,.4)(.2,.6)
	\psline[linewidth=.8pt](.2,.6)(.25,.55)
	\psline[linewidth=.8pt](.4,0)(.6,.2)
	\psline[linewidth=.8pt](.6,.2)(.55,.25)
	\psline[linewidth=.8pt](.4,0)(0,0)(0,.4)
	\psline[linewidth=.8pt, linestyle=dotted](1,0)(1,1)(0,1)
	\psline[linewidth=.8pt, linestyle=dotted](.6,.2)(1,.6)
	\psline[linewidth=.8pt, linestyle=dotted](.2,.6)(.6,1)
	\psline[linewidth=.8pt, linestyle=dotted](0,.4)(.4,.4)(.4,0)
	\psline[linewidth=.8pt, linestyle=dotted](.25,.55)(1,1)
	\psline[linewidth=.8pt, linestyle=dotted](.55,.25)(1,1)
	\psdots[dotstyle=|](1,0)(.4,0)
	\psdots[dotstyle=|, dotangle=90](0,1)(0,.4)
	\uput[225](0,0){$O$}
	\uput[315](0,.4){$A$}
	\uput[45](.25,.55){$B$}
	\uput[90](.2,.6){$B'$}
	\uput[45](.4,.4){$C$}
	\uput[45](.55,.25){$D$}
	\uput[0](.6,.2){$D'$}
	\uput[135](.4,0){$E$}
	\uput[45](1,1){$F$}
\end{pspicture}
}
\caption{Necessary and sufficient conditions on the exponents $\theta$ and $\wt{\theta}$ for global inhomogeneous Strichartz estimates.}
\label{fig:sharpness of theta against theta}
\end{figure}

\begin{proof} Suppose that $\sigma>0$ and that the global inhomogeneous Strichartz estimate (\ref{eq:str3}) holds for any $\{U(t):t\geq0\}$ satisfying the energy estimate (\ref{est:energy}) and the dispersive estimate (\ref{est:untrunc_decay}). We systemically establish the necessity of each of the conditions above.

Recall that (\ref{est:energy}) and (\ref{est:untrunc_decay}) are invariant with respect to scaling (\ref{scaling}). When the same scaling is applied to (\ref{eq:str3}), we obtain
\begin{multline*}
\lambda^{\sigma\theta/2+1+1/q}\norm{(TT^*)_RF}\stn{q}{\Bth^*}\lesssim\lambda^{\sigma\wt{\theta}/2+1/\wt{q}'}\norm{F}\stn{\wt{q}'}{\Bwth}\\
\forall F\in\st{\wt{q}'}{\Bwth}\cap\st{1}{\B_0}.
\end{multline*}
Invariance with respect to scaling requires that
\[\frac{\sigma\theta}{2}+1+\frac{1}{q}=\frac{\sigma\wt{\theta}}{2}+\frac{1}{\wt{q}'},\]
which is equivalent to (\ref{eq:nec1}).

To show the necessity of condition (\ref{eq:nec2}), consider any family $\{U(t):t\in\R\}$ which possesses the group property $U(t)U(s)^*=U(t-s)$ whenever $s$ and $t$ are real numbers (the Schr\"odinger group given by $U(t)=e^{it\Delta}$ will suffice). Under such circumstances, the operator $(TT^*)_R:\st{\wt{q'}}{\Bwth}\to\st{q'}{\Bth^*}$ is translation invariant (by which we mean that it commutes with all time translation operators), implying that $\wt{q}'\leq q$ by a vector-valued version of \cite[Theorem 1.1]{lH60}. This last inequality is equivalent to (\ref{eq:nec2}).

The necessity of (\ref{eq:nec3}) and (\ref{eq:nec4}) is the result of two particular forcing terms $F$ constructed for the Schr\"odinger group (see \cite[Examples 6.9 and 6.10]{dF05a}, or alternately \cite[Section 3]{mV07} for details).

Finally, the exclusion of the case $(\theta,\wt{\theta},\sigma)=(1,1,1)$ follows from the negative result of T. Tao \cite{tT2000} for the Schr\"odinger equation in two spatial dimensions.
\end{proof}

\section{Application to the Schr\"odinger equation with potential}\label{s:schrodinger}

In this section we show how Theorem \ref{th:inhomogeneous} is used to obtain Strichartz estimates for various Schr\"odinger equations. First we consider the standard Schr\"odinger equation (that is, without potential) and show that Theorem \ref{th:inhomogeneous} recovers the Strichartz estimates obtained by Foschi \cite{dF05a} and Vilela \cite{mV07}. After this we obtain Strichartz estimates for Schr\"odinger equations with potential (see Corollary \ref{cor:strichartz for schrodinger with potential}); these cannot be deduced from the results of \cite{dF05a} and \cite{mV07} (see Remark \ref{rem:foschi insufficient} for further details).

First we need a result about real interpolation of $L^p$ spaces. Suppose that $p_0,p_1\in[1,\infty]$, $p_0\neq p_1$, $\min(p_0,p_1)<q\leq\infty$ and $0<\theta<1$. If $1/p=(1-\theta)/p_0+\theta p_1$ then
\[\big(L^{p_0}(\R^n),L^{p_1}(\R^n)\big)_{\theta,q}=L^{p,q}(\R^n),\]
where $L^{p,q}(\R^n)$ denotes the Lorentz space (with exponents $p$ and $q$) on $\R^n$ (see \cite[Theorem 5.2.1]{BL76}). Moreover, we have the continuous embedding
\[L^p(\R^n)\subseteq\big(L^{p_0}(\R^n),L^{p_1}(\R^n)\big)_{\theta,q}=L^{p,q}(\R^n)\]
whenever $p\leq q$ (see \cite[p. 2]{BL76}).

Suppose that $n$ is a positive integer. We say that a pair $(q,r)$ of Lebesgue exponents are \textit{Schr\"odinger n-acceptable} if either
\[1\leq q<\infty,\qquad 2\leq r\leq\infty,\qquad\frac{1}{q}<n\left(\frac{1}{2}-\frac{1}{r}\right)\]
or $(q,r)=(\infty,2)$.

\begin{corollary}[Foschi \cite{dF05a}, Vilela \cite{mV07}]\label{cor:inhomogeneous strichartz for schrodinger}
Suppose that $n$ is a positive integer and that the exponent pairs $(q,r)$ and $(\wt{q},\wt{r})$ are Schr\"odinger $n$-acceptable, satisfy the scaling condition
\[\frac{1}{q}+\frac{1}{\wt{q}}=\frac{n}{2}\left(1-\frac{1}{r}-\frac{1}{\wt{r}}\right)\]
and either the conditions
\[\frac{1}{q}+\frac{1}{\wt{q}}<1,\qquad\frac{n-2}{r}\leq\frac{n}{\wt{r}},\qquad\frac{n-2}{\wt{r}}\leq\frac{n}{r}\]
or the conditions
\[\frac{1}{q}+\frac{1}{\wt{q}}=1,\qquad\frac{n-2}{r}<\frac{n}{\wt{r}},\qquad\frac{n-2}{\wt{r}}<\frac{n}{r},
\qquad\frac{1}{r}\leq\frac{1}{q},\qquad\frac{1}{\wt{r}}\leq\frac{1}{\wt{q}}.\]
When $n=2$ we also require that $r<\infty$ and $\wt{r}<\infty$.
If $F\in\st{\wt{q}'}{L^{\wt{r}'}(\R^n)}$ and $u$ is a weak solution of the inhomogeneous Schr\"odinger equation
\[iu'(t)+\Delta u(t)=F(t),\qquad u(0)=0\]
then
\begin{equation}\label{est:inhomogeneous strichartz for schrodinger}
\norm{u}\stn{q}{L^r(\R^n)}\lesssim\norm{F}\stn{\wt{q}'}{L^{\wt{r}'}(\R^n)}.
\end{equation}
\end{corollary}

\begin{proof}
This is a simple application of Theorem \ref{th:inhomogeneous} when $\Hi=L^2(\R^n)$, $(\B_0,\B_1)=(L^2(\R^n),L^1(\R^n))$, $\sigma=n/2$ and $U(t)=e^{it\Delta}$. That the energy estimate is satisfied follows from Plancherel's theorem, while the dispersive estimate follows from a simple bound on the integral representation of $e^{it\Delta}$ (see, for example, \cite[Section 6]{dF05a} for details). To obtain (\ref{est:inhomogeneous strichartz for schrodinger}) from (\ref{eq:str4}), we use the embedding $L^{r'}(\R^n)\subseteq L^{r',q'}(\R^n)$ whenever $r'\leq q'$.
\end{proof}

We now show that our generalisation of Foschi's work \cite{dF05a} allows one to obtain new Strichartz estimates for Schr\"odinger equations involving certain potentials.

Suppose that $V:\R^3\to\R$ is a real-valued potential on $\R^3$ with decay
\begin{equation}\label{est:potential}
|V(x)|\leq C\nn<x>^{-\beta}\qquad\forall x\in\R^3,
\end{equation}
where $\beta>5/2$ and $\nn<x>=(1+|x|^2)^{1/2}$. Consider the Hamiltonian operator $H$, given by $H=-\Delta+V$, on the Hilbert space $L^2(\R^3)$ with domain $W^{2,2}(\R^3)$, where $W^{k,p}(X)$ denotes the Sobolev space of order $k$ in $L^p(X)$. Our goal is to obtain spacetime estimates for the solution $u$ of the inhomogeneous initial value problem
\begin{equation}\label{eq:NLS IVP}
\begin{cases}
\left(i\frac{\partial}{\partial t}+H\right)u(t)=F(t)\qquad\forall t\in[0,\tau],\\
u(0)=f,
\end{cases}
\end{equation}
where $\tau>0$ and, for each time $t$ in $\R$, $f$ and $F(t)$ are complex-valued functions on $\R^3$.

Hamiltonians that satisfy the above conditions are considered by K. Yajima in \cite{kY05}. There it mentions that $H$ is self-adjoint on $L^2(\R^3)$ with a spectrum consisting of a finite number of nonpositive eigenvalues, each of finite multiplicity, and the absolutely continuous part $[0,\infty)$. Denote by $P_c$ the orthogonal projection from $L^2(\R^3)$ onto the continuous spectral subspace for $H$. Under the general assumption (\ref{est:potential}), it is known that $P_c$, when viewed as an operator on $L^p(\R^3)$, is bounded only when $2/3<p<3$.

Denote by $\mathcal{H}_{\gamma}$ the weighted Lebesgue space $L^2(\R^3,\nn<x>^{2\gamma}\dd x)$. When $\gamma\in(1/2,\beta-1/2)$, define the null space $\mathcal{N}$ by
\[\mathcal{N}=\left\{\phi\in\mathcal{H}_{-\gamma}:\phi(x)+\frac{1}{4\pi}\int_{\R^3}
\frac{V(y)\phi(y)}{|x-y|}\,\dd y=0\right\}.\]
As noted in \cite{kY05}, the space $\mathcal{N}$ is finite dimensional and is independent of the choice of $\gamma$ in the interval $(1/2,\beta-1/2)$. All $\phi$ belonging to $\mathcal{N}$ satisfy the stationary Schr\"odinger equation
\begin{equation}\label{eq:free schrodinger}
-\Delta\phi(x)+V(x)\phi(x)=0,
\end{equation}
where (\ref{eq:free schrodinger}) is to be interpreted in the distributional sense. Conversely, any function $\phi\in\mathcal{H}_{-3/2}$ which satisfies (\ref{eq:free schrodinger}) belongs to $\mathcal N$. Hence, if $0$ is an eigenvalue of $H$ with associated eigenspace $\mathcal E$, then $\mathcal E$ is a subspace of $\mathcal N$.

\begin{definition}
We say that $H$ or $V$ is of \textit{generic type} if $\mathcal{N}=\{0\}$ and is of \textit{exceptional type} otherwise. The Hamiltonian $H$ is of \textit{exceptional type of the first kind} if $\mathcal{N}\neq\{0\}$ and $0$ is not an eigenvalue of $H$. It is of \textit{exceptional type of the second kind} if $\mathcal{E}=\mathcal{N}\neq\{0\}$. Finally, we say that $H$ is of \textit{exceptional type of the third kind} if $\{0\}\subset\mathcal{E}\subset\mathcal{N}$ with strict inclusions.
\end{definition}

While most $V$ are of generic type, examples that are of exceptional type are interesting from a physical point of view. In particular, if $V$ is of exceptional of the third kind then any function $\phi$ in $\mathcal{N}\bs\mathcal{E}$ is called a \textit{resonance} of $H$.

We would like to apply Theorems \ref{th:KT98} and \ref{th:inhomogeneous} to the case where $U(t)$ is the operator $e^{itH}$, defined by the functional calculus for self-adjoint operators. However, if $g$ is an eigenfunction of $H$ with corresponding eigenvalue $\lambda$, then
\begin{equation}\label{eq:static eigenfunction}
U(s)U(t)^*g=e^{i(s-t)H}g=e^{i(s-t)\lambda}g
\end{equation}
and therefore $U(s)U(t)^*g$ is stationary. Consequently, the dispersive hypothesis (\ref{est:untrunc_decay}) is not satisfied. Fortunately, this is not the case if $g$ lies in the continuous spectral subspace of $H$.

\begin{theorem}[K. Yajima \cite{kY05}]\label{th:dispersive schrodinger}
There exists a positive constant $C_p$ such that the dispersive estimate
\begin{equation}\label{est:dispersive schrodinger}
\norm{e^{itH}P_cg}_{p'}\leq C_p|t|^{-3(1/p-1/2)}\norm{g}_p\qquad\forall g\in L^2(\R^3)\cap L^p(\R^3)
\quad\forall\mbox{ real }t\neq0.
\end{equation}
is satisfied in the following two cases:
\begin{enumerate}
\item[(i)] if $H$ is of generic type, $\beta>5/2$ and $1\leq p\leq2$; and
\item[(ii)] if $H$ is of exceptional type, $\beta>11/2$ and $3/2<p\leq2$.
\end{enumerate}
\end{theorem}

\begin{remark}\label{rem:foschi insufficient}
If $H$ is of exceptional type then (\ref{est:dispersive schrodinger}) \textit{cannot} hold when $p=1$, otherwise it would contradict the local decay estimate of Jensen--Kato \cite{JK79} or Murata \cite{mM82}. Hence one cannot apply the results of Foschi \cite{dF05a} to this situation.
\end{remark}

Our immediate goal is to apply Theorems \ref{th:KT98} and \ref{th:inhomogeneous} to the continuous spectral subspace of $H$. If $u$ is a solution to (\ref{eq:NLS IVP}), define $u_c$ by $u_c(t)=P_cu(t)$ for all $t$ in $[0,\tau]$. Similarly, let $P_{pp}$ denote the orthogonal projection onto the pure-point spectral subspace of $H$ and define $u_{pp}$ by $u_{pp}(t)=P_{pp}u(t)$ for all $t$ in $[0,\tau]$. It is clear that $u=u_{pp}+u_c$. 

The dispersive estimate (\ref{est:dispersive schrodinger}) gives rise to the admissibility conditions
\begin{equation}\label{eq:potental admissibility}
\frac{1}{q}+\frac{3}{2r}=\frac{3}{4},\quad4<q\leq\infty;\qquad
\frac{1}{\wt{q}}+\frac{3}{2\wt{r}}=\frac{3}{4},\quad4<\wt{q}\leq\infty
\end{equation}
sketched in Figure \ref{fig:potential admissibility}.
These correspond to the sharp $\sigma$-admissibility conditions in the case when $\sigma=3(1/p-1/2)$, $\Hi=\B_0=L^2(\R^3)$, $\B_1=L^p(\R^3)$ and $p\to3/2$ from above. Note that they also correspond to the Schr\"odinger admissibility conditions (\ref{eq:schrodinger admissible}) when $n=3$, but with restricted range.

When considering the inhomogeneous problem with zero initial data, the exponent conditions of Theorem \ref{th:inhomogeneous} reduce to the scaling condition
\begin{equation}\label{eq:potential scaling condition}
\frac{1}{q}+\frac{1}{\wt{q}}=\frac{3}{2}\left(1-\frac{1}{r}-\frac{1}{\wt{r}}\right)
\end{equation}
and the acceptability conditions
\begin{align}\label{eq:potential acceptability1}
1\leq q<\infty,\quad 2\leq r<3,\quad \frac{1}{q}<3\left(\frac{1}{2}-\frac{1}{r}\right),\quad\mbox{or }(q,r)=(\infty,2);\\
1\leq \wt{q}<\infty,\quad 2\leq\wt{r}<3,\quad\frac{1}{\wt{q}}<3\left(\frac{1}{2}-\frac{1}{\wt{r}}\right)
,\quad\mbox{or }(\wt{q},\wt{r})=(\infty,2).
\label{eq:potential acceptability2}
\end{align}
This is because $\sigma=3(1/p-1/2)<1$.

\begin{figure}
\centering

\begin{pspicture}(-.5,-.5)(4,4)
	\psset{unit=1cm}
	\newgray{gray2}{.8}
	\pspolygon[linewidth=.8pt, linecolor=white, fillcolor=gray2, fillstyle=solid](3,0)(2,0)(2,3)
	\psline[linewidth=.4pt]{->}(0,0)(0,3.5)
	\psline[linewidth=.4pt]{-o}(0,0)(2,0)
	\psline[linewidth=.4pt]{*->}(3,0)(3.5,0)
	\uput[200](0,3){$\frac{1}{2}$}
	\uput[d](3,0){$\frac{1}{2}$}
	\uput[d](3.5,0){$\frac{1}{r}$}
	\uput[160](0,3.5){$\frac{1}{q}$}
	\uput[d](2,0){$\frac{1}{3}$}
	\uput[l](0,1.5){$\frac{1}{4}$}
	\psline[linewidth=.8pt, linestyle=dashed]{*-o}(3,0)(2,3)
	\psline[linewidth=.8pt, linestyle=dashed]{o-o}(2,3)(2,0)
	\psline[linewidth=.8pt, linestyle=dashed]{*-o}(3,0)(2,0)
	\psline[linewidth=.8pt]{*-o}(3,0)(2,1.5)
	\psdots[dotstyle=|, dotangle=90](0,1.5)(0,3)
	\uput[225](0,0){$0$}
	\uput[55](3,0){$A$}
	\uput[180](2,1.5){$B$}
\end{pspicture}

\caption{The line segment $AB$ and the shaded region respectively give admissible and acceptable exponents for Strichartz estimates associated to the inhomogeneous initial value problem (\ref{eq:NLS IVP}).}
\label{fig:potential admissibility}
\end{figure}
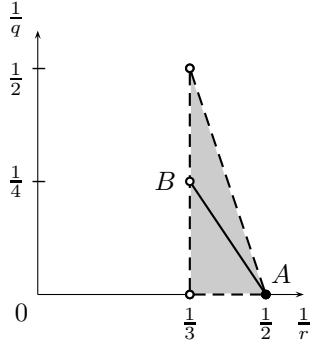

\begin{corollary}\label{cor:strichartz for u_c}
Suppose that $u$ is a (weak) solution to problem (\ref{eq:NLS IVP}) for some data $f$ in $L^2(\R^3)$, some source $F$ and for some time $\tau$ in $(0,\infty)$.
\begin{enumerate}
\item[(i)] If $(q,r)$ and $(\wt{q},\wt{r})$ satisfy the admissibility condition (\ref{eq:potental admissibility}) and $F$ belongs to $L^{\wt{q}'}([0,\tau];L^{\wt{r}'}(\R^3))$, then
\begin{equation}
\norm{u_c}_{L^q([0,\tau],L^r(\R^3))}\lesssim\norm{f}_{L^2(\R^3)}+\norm{F}_{L^{\wt{q}'}([0,\tau],L^{\wt{r}'}(\R^3))}.
\end{equation}
\item[(ii)] If the exponent pairs $(q,r)$ and $(\wt{q},\wt{r})$ satisfy conditions (\ref{eq:potential scaling condition}), (\ref{eq:potential acceptability1}) and (\ref{eq:potential acceptability2}), $f=0$ and $F\in L^{\wt{q}'}([0,\tau];L^{\wt{r}'}(\R^3))$, then
\[\norm{u_c}_{L^q([0,\tau],L^r(\R^3))}\lesssim\norm{F}_{L^{\wt{q}'}([0,\tau],L^{\wt{r}'}(\R^3))}.\]
\end{enumerate}
\end{corollary}

\begin{proof}
Fix $p$ such that
\[3/2<p<\min\{r',\wt{r}'\}.\]
For $t$ in $\R$ define $U(t)$ on $L^2(\R^3)$ by $U(t)=1_{[0,\tau]}(t)e^{itH}P_c$. If $g$ belongs to $L^2(\R^3)\cap L^p(\R^3)$, then
\begin{align*}
\norm{U(s)U(t)^*g}_{p'}
&\leq \norm{e^{i(s-t)H}P_cg}_{p'}\\
&\leq |s-t|^{-3(1/p-1/2)}\norm{g}_p.
\end{align*}
by Theorem \ref{th:dispersive schrodinger}.
Therefore $\{U(t):t\in\R\}$ satisfies the dispersive estimate (\ref{est:untrunc_decay}) when $\sigma=3(1/p-1/2)$,
$\B_0=\Hi=L^2(\R^3)$ and $\B_1=L^p(\R^3)$.
Moreover, since each operator $e^{-itH}$ on $L^2(\R^3)$ is unitary and $P_c$ is an orthogonal projection, $\{U(t):t\in\R\}$ also satisfies the energy estimate (\ref{est:energy}). Now if $u$ is a weak solution to (\ref{eq:NLS IVP}) then
\begin{equation}\label{eq:NLS solution}
u(t)=e^{itH}f-i\int_0^te^{i(t-s)H}F(s)\,\dd s
\end{equation}
by Duhamel's principle and the functional calculus for self-adjoint operators. Hence
\begin{align*}
u_c(t)
&=e^{itH}P_cf-i\int_0^te^{i(t-s)H}P_cF(s)\,\dd s\\
&=Tf(t)-i(TT^*)_RF(t).
\end{align*}
An application of Theorem \ref{th:KT98} and Theorem \ref{th:inhomogeneous} gives the required spacetime estimates for $u_c$ once we observe that
\[L^{r'}(\R^3)\subseteq L^{r',2}(\R^3)=\Bth,\]
where $1/r'=(1-\theta)/2+\theta/p$ and the inclusion is continuous.
\end{proof}

To find a spacetime estimate for the solution $u$ of (\ref{eq:NLS IVP}), we now need only analyse the projection of each $u(t)$ onto the pure point spectral subspace of $H$. It is known (see \cite[p.~477]{kY05}) that eigenfunctions of $H$ with negative eigenvalues decay at least exponentially. Since such eigenfunctions belong to the domain of $\Delta$, they are necessarily continuous and consequently also belong to $L^r(\R^3)$ whenever $1\leq r\leq\infty$ by Sobolev embedding. However, if $0$ is an eigenvalue then a corresponding eigenfunction $\phi$ may decay as slowly as $C\langle x\rangle^{-2}$ when $|x|\to\infty$. Hence, in general, $\phi$ is a member of $L^p(\R^3)$ only when $p>3/2$.

Except in the case when the time exponent is $\infty$, one cannot hope for a spacetime estimate for $u_{pp}$ which is global in time due to (\ref{eq:static eigenfunction}). However, one can still obtain spacetime estimates on finite time intervals. For illustrative purposes, the next lemma gives a crude spacetime estimate for $u_{pp}$ when $H$ is of exceptional type. No further analysis on $H$ is needed.

\begin{lemma}\label{lemma:strichartz for u_pp}
Suppose that $\tau>0$, that $q,\wt{q}\in[1,\infty]$ and that $r,\wt{r}\in(3/2,3)$. Suppose also that $f\in L^2(\R^3)$, $F\in L^{\wt{q}'}([0,\tau],L^{\wt{r}'}(\R^3))$ and $H$ is of exceptional type. If $u$ is a (weak) solution to problem (\ref{eq:NLS IVP}) then
\[\norm{u_{pp}}_{L^q([0,\tau],L^r(\R^3))}\leq C_{r,H}\left(\norm{P_{pp}f}_2+\tau^{1/q+1/\wt{q}}\norm{P_{pp}F}_{L^{\wt{q}'}([0,\tau],L^{\wt{r}'}(\R^3))}\right)\]
where the positive constant $C_{r,H}$ depends on $r$ and $H$ only. If $q=\wt{q}=\infty$ then $\tau^{1/q+1/\wt{q}}$ is interpreted as $1$.
\end{lemma}

\begin{proof}
Suppose that $\{\phi_j:j=1,\ldots,n\}$ is a complete orthonormal set of eigenfunctions for $H$ on $L^2(\R^3)$ corresponding to the set $\{\lambda_j:j=1,\ldots,n\}$ of eigenvalues (counting multiplicities). Write
\[P_{pp}f=\sum_{i=1}^n\alpha_j\phi_j \qquad\mbox{and}\qquad P_{pp}F(s)=\sum_{j=1}^n\beta_j(s)\phi_j,\]
where each $\alpha_j$ and $\beta_j(s)$ is a complex scalar. By orthogonality and the equivalence of norms in finite dimensional normed spaces (see \cite[p.~69]{jC90}), there are positive constants $C$ and $C'$ (both independent of $f$, $F(s)$, $\{\alpha_j\}$ and $\{\beta_j(s)\}$) such that
\[\sum_{j=1}^n|\alpha_j|\leq C\Big(\sum_{j=1}^n|\alpha_j|^2\Big)^{1/2}=C\norm{P_{pp}f}_2\]
and
\[\sum_{j=1}^n|\beta_j(s)|\leq C\Big(\sum_{j=1}^n|\beta_j(s)|^2\Big)^{1/2}=C\norm{P_{pp}F(s)}_2\leq C'\norm{P_{pp}F(s)}_{\wt{r}'}.\]
Following from (\ref{eq:NLS solution}),
\begin{align*}
u_{pp}(t)
&=e^{itH}P_{pp}f-i\int_0^te^{i(t-s)H}P_{pp}F(s)\,\dd s\\
&=\sum_{j=1}^n\alpha_je^{it\lambda_j}\phi_j-i\int_0^t\sum_{j=1}^n\beta_j(s)e^{i(t-s)\lambda_j}\phi_j\,\dd s.
\end{align*}
By taking the $L^q([0,\tau],L^r(\R^3))$ norm and applying H\"older's inequality,
\begin{align*}
\norm{u_{pp}}_{L^q([0,\tau],L^r(\R^3))}
&\leq \sum_{i=j}^n|\alpha_j|\norm{\phi_j}_r+
\norm{\int_0^t\sum_{j=1}^n|\beta_j(s)|\norm{\phi_j}_r\,\dd s}_{L^q([0,\tau])}\\
&\leq C''\max_{1\leq j\leq n}\norm{\phi_j}_r
\left(\norm{P_{pp}f}_2+\norm{\int_0^t\norm{P_{pp}F(s)}_{\wt{r}'}\,\dd s}_{L^q([0,\tau])}\right)\\
&\leq C_{r,H}\left(\norm{P_{pp}f}_2+\tau^{1/q}\norm{P_{pp}F}_{L^1([0,\tau];L^{\wt{r}'}(\R^3))}\right)\\
&\leq C_{r,H}\left(\norm{P_{pp}f}_2+\tau^{1/q+1/\wt{q}}\norm{P_{pp}F}_{L^{\wt{q}'}([0,\tau];L^{\wt{r}'}(\R^3))}\right)
\end{align*}
where
\[C_{r,H}=C''\max_{1\leq j\leq n}\norm{\phi_j}_r.\]
This completes the proof.
\end{proof}

Combining the lemma with Corollary \ref{cor:strichartz for u_c} and the fact that $u=u_c+u_{pp}$ gives the following result.

\begin{corollary}\label{cor:strichartz for schrodinger with potential}
Suppose that $H$ is of exceptional type, that $\tau>0$ and that $(q,r)$ and $(\wt{q},\wt{r})$ satisfy the admissibility conditions (\ref{eq:potental admissibility}).
If $f\in L^2(\R^3)$ and $F\in L^{\wt{q}'}([0,\tau],L^{\wt{r}'}(\R^3))$ and $u$ is a (weak) solution to problem (\ref{eq:NLS IVP}) then
\begin{equation}\label{est:schrodinger strichartz with potential}
\norm{u}_{L_t^q([0,\tau],L^r(\R^3))}
\lesssim
\norm{f}_{L^2(\R^3)}+\big(1+\tau^{1/q+1/\wt{q}}\big)\norm{F}_{L^{\wt{q}'}_t([0,\tau],L^{\wt{r}'}(\R^3))}.
\end{equation}
If $q=\wt{q}=\infty$ then $\tau^{1/q+1/\wt{q}}$ is interpreted as $1$. If $f=0$ then the conditions on $(q,r)$ and $(\wt{q},\wt{r})$ may be relaxed to satisfying (\ref{eq:potential scaling condition}), (\ref{eq:potential acceptability1}) and (\ref{eq:potential acceptability2}).
\end{corollary}

\section{Application to the wave equation}\label{s:wave}

In this section we consider the wave equation. Keel and Tao \cite{KT98} indicated that Theorem \ref{th:KT98} could be used to obtain Strichartz estimates (in Besov norms), following a similar approach to \cite{GV95}, but did not show any details. We shall therefore do so here, presenting the results in Corollary \ref{cor:wave strichartz}. These estimates coincide with those of \cite{GV95}, except that the so-called `endpoint' estimate is now included (see the point $Q$ in Figure \ref{fig:wave-acceptable region} (a)). Next we apply Theorem \ref{th:inhomogeneous} to the wave equation, thus obtaining a new set of inhomogeneous Strichartz estimates for the wave equation (see Corollaries \ref{cor:inhomogeneous strichartz for wave} and \ref{cor:inhomogeneous strichartz for wave sharp}). Finally, we indicate how a small modification of these arguments obtains Strichartz estimates for the Klein--Gordon equation (see Remark \ref{rem:KGE}).

We begin with a brief introduction to homogeneous Besov spaces; for a treatment at greater depth, the reader is referred to \cite{BL76} and \cite{hT78}. To start, we consider Littlewood--Paley dyadic decompositions. Suppose that $\psi\in C^{\infty}_0(\R^n)$ and that its Fourier transform $\hat{\psi}$ satisfies the properties $0\leq\hat{\psi}\leq1$, $\hat{\psi}(\xi)=1$ whenever $|\xi|\leq1$ and $\hat{\psi}(\xi)=0$ whenever $|\xi|\geq2$. If $j\in\Z$ then define $\varphi_j$ by $\hat{\varphi}_0(\xi)=\hat{\psi}(\xi)-\hat{\psi}(2\xi)$ and $\hat{\varphi}_j(\xi)=\hat{\varphi}_0(2^{-j}\xi)$. This means that
\[
\mathrm{supp}(\hat{\varphi}_j)\subseteq\{\xi\in\R^n:2^{j-1}\leq|\xi|\leq2^{j+1}\}
\qquad\mbox{and}\qquad
\sum_{j\in\Z}\hat{\varphi}_j(\xi)=1
\]
for any $\xi$ in $\R^n\bs\{0\}$, with at most two nonvanishing terms in the sum. When $j\in\Z$, define $\wt{\varphi}_j$ by the formula
\begin{equation}\label{eq:varphi_j}
\wt{\varphi}_j=\varphi_{j-1}+\varphi_j+\varphi_{j+1}
\end{equation}
This implies that $\hat{\varphi}_j=\hat{\wt{\varphi}}_j\hat{\varphi}_j$, thereby allowing for the use of the standard trick
\begin{equation}\label{eq:convolution trick}
\varphi_j*u=\wt{\varphi}_j*\varphi_j*u
\end{equation}
for any tempered distribution $u$.

If $\rho\in\R$, $1\leq r\leq\infty$, $1\leq s\leq\infty$ and $u$ is a tempered distribution then define $\norm{u}\besn{\rho}{r}{s}$ by
\begin{equation}\label{eq:besov norm}
\norm{u}\besn{\rho}{r}{s}=\left(\sum_{j\in\Z}2^{\rho j}\norm{\varphi*u}_{L^r(\R^n)}^s\right)^{1/s}.
\end{equation}
We note that if $u$ is a polynomial then $\mathrm{supp}(\hat{u})=\{0\}$ and hence $\norm{u}\besn{\rho}{r}{s}=0$. Conversely if $\norm{u}\besn{\rho}{r}{s}=0$ then $u$ is a polynomial. We therefore define the homogeneous Besov space $\bes{\rho}{r}{s}$ to be the completion in $\norm{\cdot}\besn{\rho}{r}{s}$ of the set of equivalence classes of tempered distributions $u$, modulo polynomials, such that $\norm{u}\besn{\rho}{r}{s}<\infty$.

We now consider real interpolation of homogeneous Besov spaces. In what follows, $\mathcal{S}'(\R^n)$ denotes the set of tempered distributions on $\R^n$, $\mathcal{P}(\R^n)$ the set of polynomials on $\R^n$ and $L^{r,s}(\R^n)$ the Lorentz space on $\R^n$.

\begin{lemma}\label{lem:besov interpolation}\cite[Section 2.4]{hT78}
Suppose that $\rho_0,\rho_1\in\R$, $\rho_0\neq\rho_1$, $r_0,r_1\in[1,\infty)$, $r_0\neq r_1$, $s_0,s_1\in(1,\infty)$ and $\theta\in(0,1)$. Then
\[\Big(\bes{\rho_0}{r_0}{s_0},\bes{\rho_0}{r_0}{s_0}\Big)_{\theta,s}=\bes{\rho}{r}{s,(s)}\]
where
\[\rho=(1-\theta)\rho_0+\theta\rho_1,\qquad\frac{1}{r}=\frac{1-\theta}{r_0}+\frac{\theta}{r_1},
\qquad\frac{1}{s}=\frac{1-\theta}{s_0}+\frac{\theta}{s_1}\]
and
\[\bes{\rho}{r}{s,(s)}=\left\{u\in\mathcal{S}'(\R^n)\bs\mathcal{P}(\R^n):
\norm{\big\{2^{\rho j}\norm{\varphi_j*u}_{L^{r,s}(\R^n)}\big\}_{j\in\Z}}_{\ell^s}<\infty\right\}.\]
\end{lemma}

If $s=2$ then we can use the above lemma together with the continuous embedding $L^p=L^{p,p}\subseteq L^{p,2}$ (whenever $p\leq2$) to obtain the continuous inclusion
\begin{equation}\label{eq:real interpolation of besov}
\bes{\rho}{r}{2}\subseteq\bes{\rho}{r}{2,(2)}=\Big(\bes{\rho_0}{r_0}{2},\bes{\rho_0}{r_0}{2}\Big)_{\theta,s}
\end{equation}
whenever $r\leq2$. Two other continuous embedding results will be needed. The first is an easy consequence of Young's inequality and the definition of homogeneous Besov spaces. 

\begin{lemma}\label{lem:besov embedding}\cite[Section 6.5]{BL76}
Suppose that $1\leq r_2\leq r_1\leq\infty$, $1\leq s\leq\infty$, $\rho_1,\rho_2\in\R$ and $\rho_1-n/r_1=\rho_2-n/r_2$. Then $\bes{\rho_2}{r_2}{s}\subseteq\bes{\rho_1}{r_1}{s}$ and
\begin{equation}\label{eq:besov embedding}
\norm{u}\besn{\rho_1}{r_1}{s}\leq C\norm{u}\besn{\rho_2}{r_2}{s}
\end{equation}
for some positive constant $C$.
\end{lemma}

The second involves homogeneous Sobelev spaces $\sob{\rho}{r}$, which may be defined in terms of Riesz potentials. Briefly, whenever $1<r<\infty$ and $\rho\in\R$, the space $\dot{H}_r^{\rho}$ coincides with the space $(-\Delta)^{-\rho/2}L^2(\R^n)$ with norm
\[\norm{u}_{\dot{H}^{\rho}}\approx\norm{(-\Delta)^{-\rho/2}u}_{L^2(\R^n)}\]
(see \cite{BL76} or \cite{hT78a} for further details). The homogeneous Besov spaces and homogeneous Sobolev spaces are related by interpolation and in particular by the continuous embeddings
\begin{equation}\label{eq:sob-bes embedding}
\bes{\rho}{r}{2}\subseteq\sob{\rho}{r}\quad\mbox{when $2\leq r<\infty$};\qquad
\bes{\rho}{r}{2}\supseteq\sob{\rho}{r}\quad\mbox{when $1<r\leq2$},
\end{equation}
whenever $\rho\in\R$. When $r=2$ it is customary to write $\dot{H}^{\rho}$ instead of $\sob{\rho}{2}$. In this case (\ref{eq:sob-bes embedding}) reduces to $\dot{H}^{\rho}=\bes{\rho}{2}{2}$.

We are now ready to apply Theorem \ref{th:KT98} to the wave equation.

\begin{corollary}\label{cor:wave strichartz} Suppose that $n\geq1$, that $\mu,\rho,\wt{\rho}\in\R$, that $q,\wt{q}\in[2,\infty]$ and that the following conditions are satisfied:
\begin{align}
q\geq2,\qquad
&\wt{q}\geq2,\notag\\
\frac{1}{q}\leq\frac{n-1}{2}\left(\frac{1}{2}-\frac{1}{r}\right),\qquad
&\frac{1}{\wt{q}}\leq\frac{n-1}{2}\left(\frac{1}{2}-\frac{1}{\wt{r}}\right),\notag\\
(q,r,n)\neq(2,\infty,3),\qquad
&(\wt{q},\wt{r},n)\neq(2,\infty,3),\notag\\
\rho+n\left(\frac{1}{2}-\frac{1}{r}\right)-\frac{1}{q}=\mu
=1&-\left(\wt{\rho}+n\left(\frac{1}{2}-\frac{1}{\wt{r}}\right)-\frac{1}{\wt{q}}\right).\label{eq:gap condition}
\end{align}
Suppose also that $f\in\dot{H}^{\mu}$, $g\in\dot{H}^{\mu-1}$ and $F\in\st{\wt{q}'}{\bes{-\wt{\rho}}{\wt{r}'}{2}}$. If $u$ is a (weak) solution to the initial value problem
\begin{equation}\label{IVP:wave}
\begin{cases}
&-u''(t)+\Delta u(t)=F(t)\\
&u(0)=f\\
&u'(0)=g
\end{cases}
\end{equation}
then
\begin{equation}\label{est:strichartz for wave}
\norm{u}\stn{q}{\bes{\rho}{r}{2}}
\lesssim\norm{f}\hsobn{\mu} + \norm{g}\hsobn{\mu-1} + \norm{F}\stn{\wt{q}'}{\bes{-\wt{\rho}}{\wt{r}'}{2}}.
\end{equation}
\end{corollary}

When $n>3$, the darker closed region of Figure \ref{fig:wave-acceptable region} (a) represents the range of exponent pairs $(q,r)$ and $(\wt{q},\wt{r})$ such that the Strichartz estimate (\ref{est:strichartz for wave}) is valid.

\begin{remark}
Corollary \ref{cor:wave strichartz} implies Strichartz estimates for spaces more familiar than the Besov spaces. By Besov--Sobolev embedding, estimate (\ref{est:strichartz for wave}) still holds when $\bes{\rho}{r}{2}$ is replaced everywhere by $\sob{\rho}{r}$ under the additional assumption that $r<\infty$ and $\wt{r}<\infty$. In fact, using Sobolev embedding, one can deduce that
\[\norm{u}\stn{q}{L^r(\R^n)}
\lesssim\norm{f}\hsobn{\mu}+\norm{g}\hsobn{\mu-1}+\norm{F}\stn{\wt{q}'}{L^{\wt{r}'}(\R^n)}\]
under the additional assumption that $r<\infty$ and $\wt{r}<\infty$. One may also replace the infinite interval $\R$ by any finite time interval $[0,\tau]$ where $\tau>0$.  See \cite[Corollary 1.3]{KT98} and \cite{GV95} for these variations.
\end{remark}

We begin with a heuristic argument to indicate how Theorem \ref{th:KT98} will be applied in this setting.
For convenience, write $\omega$ for the operator $(-\Delta)^{1/2}$. The homogeneous problem may be written as
\[v''(t)+\omega^2v(t)=0,\qquad v(0)=f,\qquad v'(0)=g,\]
with solution $v$ is given by
\[v(t)=\cos(\omega t)h_1+\sin(\omega t)h_2\]
for some functions $h_1$ and $h_2$ determined by imposing initial conditions. Hence
\[v(t)=\cos(\omega t)f+\omega^{-1}\sin(\omega t)g.\]
The inhomogeneous problem
\[-w''(t)+\Delta w(t)=F(t),\qquad w(0)=0,\qquad w'(0)=0\]
may be solved by Duhamel's principle to give
\[w(t)=\int_{s<t}\omega^{-1}\sin\big(\omega(t-s)\big)F(s)\,\dd s.\]
Define $\{U(t):t\in\R\}$ by $U(t)=e^{i\omega t}$. Then the solution $u$ to problem (\ref{IVP:wave}) can be written as
\begin{align}
u(t)
&=v(t)+w(t)\notag\\
&=\frac{1}{2}\big(U(t)+U(-t)\big)f+\omega^{-1}\frac{1}{2i}\big(U(t)-U(-t)\big)g\notag\\
&\quad+\int_{s<t}\omega^{-1}\frac{1}{2i}\big(U(t)U(s)^*-U(-t)U(-s)^*\big)F(s)\,\dd s\label{eq:wave decomposition}
\end{align}
and it is clear that if we have appropriate Strichartz estimates for the group $\{U(t):t\in\R\}$ then (\ref{est:strichartz for wave}) will follow. Hence define the operator $T$ by $Tf(t)=U(t)f$, whenever $f$ belongs to the Hilbert space $\bes{0}{2}{2}$.

\begin{lemma} Suppose that $n\geq1$ and that the triples $(q,r,\gamma)$ and $(\wt{q},\wt{r},\wt{\gamma})$ satisfy the conditions
\begin{align}
q\geq2,\qquad
&\wt{q}\geq2,\label{cond:wave1a}\\
\frac{1}{q}=\frac{n-1}{2}\left(\frac{1}{2}-\frac{1}{r}\right),\qquad
&\frac{1}{\wt{q}}=\frac{n-1}{2}\left(\frac{1}{2}-\frac{1}{\wt{r}}\right),\label{cond:wave1b}\\
\gamma=\frac{n+1}{2}\left(\frac{1}{2}-\frac{1}{r}\right),\qquad
&\wt{\gamma}=\frac{n+1}{2}\left(\frac{1}{2}-\frac{1}{\wt{r}}\right),\label{cond:wave1c}\\
(q,r,n)\neq(2,\infty,3),\qquad
&(\wt{q},\wt{r},n)\neq(2,\infty,3).\label{cond:wave1d}
\end{align}
Then the operator $T$ satisfies the Strichartz estimates
\begin{equation}\label{est:intermediate wave strichartz hom}
\norm{Tf}\stn{q}{\bes{-\gamma}{r}{2}}\lesssim\norm{f}\besn{0}{2}{2}\qquad\forall f\in\bes{0}{2}{2}
\end{equation}
and
\begin{equation}\label{est:intermediate wave strichartz inhom}
\norm{(TT^*)_RF}\stn{q}{\bes{-\gamma}{r}{2}}\lesssim\norm{F}\stn{\wt{q}'}{\bes{\wt{\gamma}}{\wt{r}'}{2}}
\qquad\forall F\in\st{\wt{q}'}{\bes{\wt{\gamma}}{\wt{r}'}{2}}.
\end{equation}
\end{lemma}

\begin{proof} We begin with the stationary phase estimate
\[\sup_{x\in\R^n}\left|\int_{\xi\in\R^n}\exp(it|\xi|+i\ip<x,\xi>)\hat{\varphi}_0(\xi)\,\dd\xi\right|
\leq C|t|^{-(n-1)/2},\]
where $C$ is a positive constant (see, for example, \cite[Section 7.7]{lH83}). For $j$ in $\Z$, apply the scaling $\xi\leftarrow 2^{-j}\xi$, $x\leftarrow 2^{j}x$, $t\leftarrow 2^{j}t$ to obtain
\[\sup_{x\in\R^n}2^{-jn}\left|\int_{\xi\in\R^n}\exp(it|\xi|+i\ip<x,\xi>)\hat{\varphi}_j(\xi)\,\dd\xi\right|\\
\leq C|t|^{-(n-1)/2}2^{-j(n-1)/2}.\]
The above estimate may be rewritten as
\[\norm{U(t)\varphi_j}_{L^{\infty}(\R^n)}\lesssim|t|^{-(n-1)/2}2^{j(n+1)/2}.\]
If $f$ is a sufficiently regular function (or distribution) in the spatial variable then
\begin{align}
\norm{\varphi_j*U(t)f}_{L^{\infty}(\R^n)}
&=\norm{\varphi_j*U(t)\wt{\varphi}_j*f}_{L^{\infty}(\R^n)}\notag\\
&\leq\norm{U(t)\varphi_j}_{L^{\infty}(\R^n)}\norm{\wt{\varphi}_j*f}_{L^1(\R^n)}\notag\\
&\lesssim|t|^{-(n-1)/2}2^{j(n+1)/2}\norm{\wt{\varphi}_j*f}_{L^1(\R^n)},\label{est:wavetrunc1}
\end{align}
by (\ref{eq:convolution trick}) and Young's inequality. Multiplying by $2^{j(n+1)/4}$ gives
\[\norm{2^{-j/2}\varphi_j*U(t)f}_{L^{\infty}(\R^n)}\lesssim|t|^{-(n-1)/2}\norm{2^{jn/2}\wt{\varphi}_j*f}_{L^1(\R^n)},\]
where the left- and right-hand sides define the $j$th term of two sequences. If we take the $\ell^2$ norm of each sequence and apply (\ref{eq:varphi_j}), then the above inequality yields
\begin{equation}\label{est:wavetrunc}
\norm{U(t)f}\besn{-(n+1)/4}{\infty}{2}\lesssim |t|^{-(n-1)/2}\norm{f}\besn{(n+1)/4}{1}{2}\qquad\forall f\in\bes{(n+1)/4}{1}{2}.
\end{equation}
This corresponds to the abstract dispersive estimate (\ref{est:untrunc_decay}).

On the other hand, each $U(t)$ is an isometry on the homogeneous Sobolev space $\hsob{0}$ and hence we have the energy estimate
\[\norm{U(t)f}\besn{0}{2}{2}\lesssim\norm{f}\besn{0}{2}{2}\qquad\forall f\in\bes{0}{2}{2},\]
by (\ref{eq:sob-bes embedding}). If $\Hi=\B_0=\bes{0}{2}{2}$ and $\B_1=\bes{(n+1)/4}{1}{2}$ then
\[\bes{\gamma}{r'}{2}\subseteq\Bth=(\B_0,\B_1)_{\theta,2}\]
by (\ref{eq:real interpolation of besov}), where $1/r'=(1-\theta)/2+\theta$ and $\gamma=(n+1)\theta/4$. It is not hard to show from here that Theorem \ref{th:KT98} proves the lemma.
\end{proof}

\begin{proof}[Proof of Corollary \ref{cor:wave strichartz}] It is well known that if $\mu\in\R$, then $\omega^{\mu}$ is an isomorphism from $\bes{\gamma}{r}{2}$ to $\bes{\gamma-\mu}{r}{2}$. Hence replacing $f$ with $\omega^{\mu}f$ in (\ref{est:intermediate wave strichartz hom}) gives
\[\norm{Tf}\stn{q}{\bes{-\gamma+\mu}{r}{2}}\lesssim\norm{f}\besn{\mu}{2}{2}\qquad\forall f\in\bes{\mu}{2}{2}.\]
The same trick applied to (\ref{est:intermediate wave strichartz inhom}) yields
\[\norm{(TT^*)_RF}\stn{q}{\bes{-\gamma+\mu}{r}{2}}\lesssim\norm{F}\stn{\wt{q}'}{\bes{\wt{\gamma}+\mu}{\wt{r}'}{2}}
\qquad\forall F\in\st{\wt{q}'}{\bes{\wt{\gamma}+\mu}{\wt{r}'}{2}}.\]
If $\rho=-\gamma+\mu$ then these estimates combine with (\ref{eq:wave decomposition}) to give
\begin{align*}
\norm{u}\stn{q}{\bes{\rho}{r}{2}}
&\lesssim\norm{Tf}\stn{q}{\bes{\rho}{r}{2}}+\norm{\omega^{-1}Tg}\stn{q}{\bes{\rho}{r}{2}}
+\norm{\omega^{-1}(TT^*)_RF}\stn{q}{\bes{\rho}{r}{2}}\\
&\lesssim\norm{f}\besn{\mu}{2}{2}+\norm{\omega^{-1}g}\besn{\mu}{2}{2}
+\norm{\omega^{-1}F}\stn{\wt{q}'}{\bes{\wt{\gamma}+\mu}{\wt{r}'}{2}}\\
&\lesssim\norm{f}\besn{\mu}{2}{2}+\norm{g}\besn{\mu-1}{2}{2}
+\norm{F}\stn{\wt{q}'}{\bes{\wt{\gamma}+\mu-1}{\wt{r}'}{2}}.
\end{align*}
If $\wt{\rho}=-(\wt{\gamma}+\mu-1)$ then the estimate above becomes
\begin{equation}\label{est:wave strichartz proof}
\norm{u}\stn{q}{\bes{\rho}{r}{2}}
\lesssim\norm{f}\hsobn{\mu} + \norm{g}\hsobn{\mu-1} + \norm{F}\stn{\wt{q}'}{\bes{-\wt{\rho}}{\wt{r}'}{2}}.
\end{equation}
So far we have imposed the conditions $\mu\in\R$, (\ref{cond:wave1a}), (\ref{cond:wave1b}), (\ref{cond:wave1d}) and
\[\rho+\frac{n+1}{2}\left(\frac{1}{2}-\frac{1}{r}\right)=\mu=1-\wt{\rho}-\frac{n+1}{2}\left(\frac{1}{2}-\frac{1}{\wt{r}}\right).\]
This last condition may be rewritten as
\[\rho+n\left(\frac{1}{2}-\frac{1}{r}\right)-\frac{1}{q}=\mu
=1-\wt{\rho}-n\left(\frac{1}{2}-\frac{1}{\wt{r}}\right)+\frac{1}{\wt{q}}.\]
Now if $r_1\geq r$ and $\rho-n/r=\rho_1-n/r_1$, then
\[\norm{u}\stn{q}{\bes{\rho_1}{r_1}{2}}\leq C\norm{u}\stn{q}{\bes{\rho}{r}{2}}\]
by Lemma \ref{lem:besov embedding}.
Similarly, if $\wt{r}_1\geq \wt{r}$ and $\wt{\rho}-n/\wt{r}=\wt{\rho}_1-n/\wt{r}_1$, then
\[\norm{F}\stn{\wt{q}'}{\bes{-\wt{\rho}}{\wt{r}'}{2}}\leq C\norm{F}\stn{\wt{q}'}{\bes{\wt{\rho}_1}{\wt{r}'_1}{2}}.\]
Applying these estimates to (\ref{est:wave strichartz proof}) gives
\begin{equation}\label{est:wave strichartz proof1}
\norm{u}\stn{q}{\bes{\rho_1}{r_1}{2}}
\lesssim\norm{f}\hsobn{\mu} + \norm{g}\hsobn{\mu-1} + \norm{F}\stn{\wt{q}'}{\bes{-\wt{\rho}_1}{\wt{r}'_1}{2}}
\end{equation}
whenever the conditions
\begin{align*}
q\geq2,\qquad
&\wt{q}\geq2,\\
\frac{1}{q}\leq\frac{n-1}{2}\left(\frac{1}{2}-\frac{1}{r_1}\right),\qquad
&\frac{1}{\wt{q}}\leq\frac{n-1}{2}\left(\frac{1}{2}-\frac{1}{\wt{r}_1}\right),\\
(q,r_1,n)\neq(2,\infty,3),\qquad
&(\wt{q},\wt{r}_1,n)\neq(2,\infty,3),\\
\rho_1+n\left(\frac{1}{2}-\frac{1}{r_1}\right)-\frac{1}{q}=\mu
=1&-\wt{\rho}_1-n\left(\frac{1}{2}-\frac{1}{\wt{r}_1}\right)+\frac{1}{\wt{q}}
\end{align*}
are satisfied. These conditions and the Strichartz estimate (\ref{est:wave strichartz proof1}) coincide with those in the statement of Corollary \ref{cor:wave strichartz}.
\end{proof}

\begin{remark} One can see from (\ref{eq:wave decomposition}) that the derivative $u'$ can also be expressed in terms of $T$, $(TT^*)_R$ and $\omega$. Thus we have the Strichartz estimate
\[\norm{u'}\stn{q}{\bes{\rho-1}{r}{2}}
\lesssim\norm{f}\hsobn{\mu} + \norm{g}\hsobn{\mu-1} + \norm{F}\stn{\wt{q}'}{\bes{-\wt{\rho}}{\wt{r}'}{2}}\]
whenever the exponents satisfy the conditions of Corollary \ref{cor:wave strichartz}.
\end{remark}

We now consider the inhomogeneous wave equation with zero initial data. Suppose that $n$ is a positive integer. We say that a pair $(q,r)$ of Lebesgue exponents are \textit{wave n-acceptable} if either
\[1\leq q<\infty,\qquad 2\leq r\leq\infty,\qquad\frac{1}{q}<(n-1)\left(\frac{1}{2}-\frac{1}{r}\right)\]
or $(q,r)=(\infty,2)$.

\begin{corollary}\label{cor:inhomogeneous strichartz for wave}
Suppose that $n$ is a positive integer and that the exponent pairs $(q,r_1)$ and $(\wt{q},\wt{r}_1)$ are wave $n$-acceptable, satisfy the scaling condition
\[\frac{1}{q}+\frac{1}{\wt{q}}=\frac{n-1}{2}\left(1-\frac{1}{r_1}-\frac{1}{\wt{r}_1}\right)\]
and the conditions
\[\frac{1}{q}+\frac{1}{\wt{q}}<1,\qquad\frac{n-3}{r_1}\leq\frac{n-1}{\wt{r}_1},\qquad\frac{n-3}{\wt{r}_1}\leq\frac{n-1}{r_1}\]
When $n=3$ we also require that $r_1<\infty$ and $\wt{r}_1<\infty$. If $r\geq r_1$, $\wt{r}\geq\wt{r}_1$, $\rho\in\R$,
\[\rho+n\left(\frac{1}{2}-\frac{1}{r}\right)-\frac{1}{q}
=1-\left(\wt{\rho}+n\left(\frac{1}{2}-\frac{1}{\wt{r}}\right)-\frac{1}{\wt{q}}\right),\]
$F\in\st{\wt{q}'}{\bes{-\wt{\rho}}{\wt{r}'}{2}}$ and $u$ is a weak solution of the inhomogeneous wave equation
\[-u''(t)+\Delta u(t)=F(t),\qquad u(0)=0,\qquad u'(0)=0,\]
then
\begin{equation}\label{est:inhomogeneous strichartz for wave}
\norm{u}\stn{q}{\bes{\rho}{r}{2}}\lesssim\norm{F}\stn{\wt{q}'}{\bes{-\wt{\rho}}{\wt{r}'}{2}}.
\end{equation}
\end{corollary}

\begin{figure}

\centering
\subfigure[]
{
\begin{pspicture}(-.3,-.3)(4.5,4.5)
	\psset{unit=4cm}
	\newgray{gray1}{.9}
	\newgray{gray2}{.75}
	\pspolygon[linewidth=1pt, fillcolor=gray2, fillstyle=solid](0,0)(.5,0)(.3,.5)(0,.5)
	\pspolygon[linecolor=white, fillcolor=gray1, fillstyle=solid](.5,0)(.3,.5)(0,.5)(0,1)(.3,1)
	\psline[linestyle=dashed, linewidth=1pt]{*-o}(.5,0)(.3,1)
	\psline[linewidth=1pt]{-o}(0,1)(.3,1)
	\psline[linewidth=1pt](0,0)(0,1)
	\psline[linewidth=1pt](0,.5)(.3,.5)
	\psline[linewidth=1pt](.5,0)(.3,.5)
	\psline[linestyle=dotted, linewidth=.4pt]{-o}(.3,0)(.3,1)
	\psline[linewidth=.4pt]{->}(0,0)(1.2,0)
	\psline[linewidth=.4pt]{->}(0,0)(0,1.2)
	\psdots[dotstyle=|](1,0)
	\uput[225](0,0){$0$}
	\uput[260](.3,0){$\frac{n-3}{2(n-1)}$}
	\uput[280](.5,0){$\frac{1}{2}$}
	\uput[d](1.2,0){$\frac{1}{r}$}
	\uput[d](1,0){$1$}
	\uput[l](0,1.2){$\frac{1}{q}$}
	\uput[l](0,1){$1$}
	\uput[l](0,.5){$\frac{1}{2}$}
	\psdots(.3,.5)
	\uput[dl](.3,.5){$Q$}
\end{pspicture}
}
\hspace*{1cm}
\subfigure[]
{
\begin{pspicture}(-.3,-.3)(4.5,4.5)
	\psset{unit=8cm}
	\newgray{gray1}{.85}
	\pspolygon[linecolor=white, fillcolor=gray1, fillstyle=solid](.5,.3)(.5,.5)(.3,.5)(.225,.375)(.375,.225)
	\psline[linestyle=dashed, linewidth=1pt]{o-o}(.225,.375)(.375,.225)
	\psline[linewidth=1pt]{o-}(.375,.225)(.5,.3)
	\psline[linewidth=1pt](.5,.3)(.5,.5)
	\psline[linewidth=1pt](.5,.5)(.3,.5)
	\psline[linewidth=1pt]{-o}(.3,.5)(.225,.375)
	\psline[linestyle=dotted, linewidth=.4pt]{-o}(0,0)(.225,.375)
	\psline[linestyle=dotted, linewidth=.4pt]{-o}(0,0)(.375,.225)
	\psline[linestyle=dotted, linewidth=.4pt](.3,0)(.3,.5)
	\psline[linestyle=dotted, linewidth=.4pt](0,.3)(.5,.3)
	\psline[linestyle=dotted, linewidth=.4pt](0,.5)(.3,.5)
	\psline[linestyle=dotted, linewidth=.4pt](.5,0)(.5,.3)
	\psline[linewidth=.4pt]{->}(0,0)(.6,0)
	\psline[linewidth=.4pt]{->}(0,0)(0,.6)
	\uput[200](0,.5){$\frac{1}{2}$}
	\uput[225](0,0){$0$}
	\uput[d](.3,0){$\frac{n-3}{2(n-1)}$}
	\uput[l](0,.3){$\frac{n-3}{2(n-1)}$}
	\uput[d](.5,0){$\frac{1}{2}$}
	\uput[d](.6,0){$\frac{1}{r_1}$}
	\uput[180](0,.6){$\frac{1}{\wt{r}_1}$}
	\uput[150](.225,.375){$C$}
	\uput[300](.375,.225){$D$}
\end{pspicture}
}
\caption{Range of exponents for Corollary \ref{cor:inhomogeneous strichartz for wave} when $n>3$.}
\label{fig:wave-acceptable region}
\end{figure}
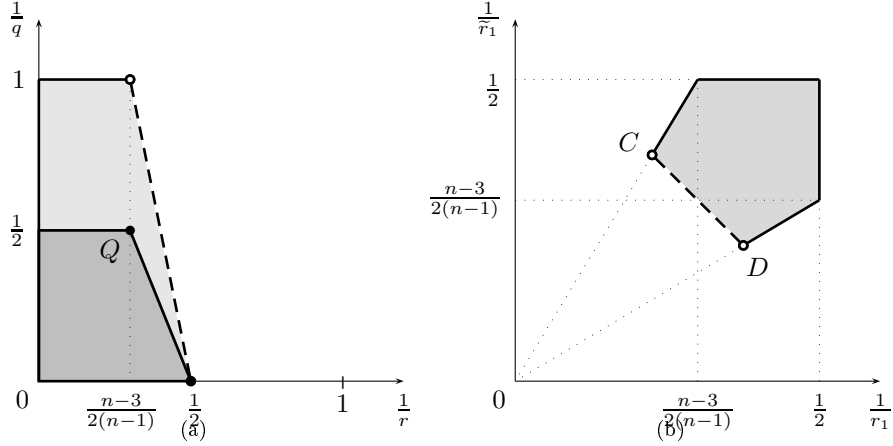

Figure \ref{fig:wave-acceptable region} shows the range for various exponents appearing in Corollary \ref{cor:inhomogeneous strichartz for wave}. In the first diagram, the dark region represents the range for the homogeneous Strichartz estimate while the union of light and dark regions represents the range for the inhomogeneous Strichartz estimate. In the second diagram, the coordinates of $C$ and $D$ are given by
\[
\left(\frac{(n-3)^2}{2(n-2)(n-1)},\frac{n-3}{2(n-2)}\right)
\quad\mbox{and}\quad
\left(\frac{n-3}{2(n-2)},\frac{(n-3)^2}{2(n-2)(n-1)}\right)
\]
respectively.

\begin{proof}
In light of the work done to prove Corollary \ref{cor:wave strichartz}, the case when $r=r_1$ and $\wt{r}=\wt{r}_1$ is a simple application of Theorem \ref{th:inhomogeneous} (i) when $(\B_0,\B_1)=(\bes{0}{2}{2},\bes{(n+1)/4}{1}{2})$ and $\sigma=(n-1)/2$. The case when $r>r_1$ and $\wt{r}>\wt{r}_1$ is obtained using the Besov embedding result of Lemma \ref{lem:besov embedding}.
\end{proof}

The case when $1/q+1/\wt{q}=1$ cannot be simply integrated into the above result via Besov embedding, except when $q=\wt{q}=2$ (which corresponds to a sharp admissible estimate obtained earlier from Theorem \ref{th:KT98}). We therefore state this case separately, using the notation $a\vee b$ and $a\wedge b$ for $\max\{a,b\}$ and $\min\{a,b\}$ respectively.

\begin{corollary}\label{cor:inhomogeneous strichartz for wave sharp}
Suppose that $n$ is a positive integer not equal to $3$ and that the exponent pairs $(q,r_1)$ and $(\wt{q},\wt{r}_1)$ are wave $n$-acceptable, satisfy the scaling condition
\[\frac{1}{q}+\frac{1}{\wt{q}}=\frac{n-1}{2}\left(1-\frac{1}{r_1}-\frac{1}{\wt{r}_1}\right)\]
and the conditions
\[\frac{1}{q}+\frac{1}{\wt{q}}=1,\qquad\frac{n-3}{r_1}<\frac{n-1}{\wt{r}_1},\qquad\frac{n-3}{\wt{r}_1}< \frac{n-1}{r_1},
\qquad\frac{1}{r_1}\leq\frac{1}{q},\qquad\frac{1}{\wt{r}_1}\leq\frac{1}{\wt{q}}.\]
If $r\geq r_1$, $\wt{r}\geq\wt{r}_1$, $\rho\in\R$,
\[\rho+n\left(\frac{1}{2}-\frac{1}{r}\right)-\frac{1}{q}
=1-\left(\wt{\rho}_1+n\left(\frac{1}{2}-\frac{1}{\wt{r}}\right)-\frac{1}{\wt{q}}\right),\]
$F\in\st{\wt{q}'}{\bes{-\wt{\rho}}{\wt{r}'}{2}}$ and $u$ is a weak solution of the inhomogeneous wave equation
\[-u''(t)+\Delta u(t)=F(t),\qquad u(0)=0,\qquad u'(0)=0,\]
then
\begin{equation}\label{est:inhomogeneous strichartz for wave sharp}
\norm{u}\stn{q}{\bes{\rho}{r}{2\vee q}}\lesssim\norm{F}\stn{\wt{q}'}{\bes{-\wt{\rho}}{\wt{r}'}{2\wedge q}}.
\end{equation}
\end{corollary}

\begin{proof}
We apply Theorem \ref{th:inhomogeneous} (ii) when $(\B_0,\B_1)=(\bes{0}{2}{2},\bes{(n+1)/4}{1}{2})$ and $\sigma=(n-1)/2$.

First suppose that $r=r_1$ and $\wt{r}=\wt{r}_1$. To obtain (\ref{est:inhomogeneous strichartz for wave sharp}) from the abstract Strichartz estimate (\ref{eq:str4}), we apply the embeddings
\[\B_{\wt{\theta},\wt{q}'}\supseteq\bes{(n+1)\wt{\theta}/4}{\wt{r}'}{2\vee\wt{q}',(\wt{q}')}
\supseteq\bes{(n+1)\wt{\theta}/4}{\wt{r}'}{2\vee\wt{q}'}\]
and
\[(\B_{\theta,q'})^*=(\bes{0}{2}{2},\bes{-(n+1)/4}{\infty}{2})_{\theta,q}
\subseteq\bes{-(n+1)\theta/4}{r}{2\wedge q,(q)}\subseteq\bes{-(n+1)\theta/4}{r}{2\wedge q}\]
(see \cite[p.~183]{hT78}, \cite[Theorem 3.7.1]{BL76} and \cite[p. 2]{BL76}) and follow the general approach of the proofs of the other corollaries in this section. Here we have taken $1/\wt{r}'=(1-\wt{\theta})/2+\wt{\theta}/1$, $1/r=(1-\theta)/2+\theta/\infty$, imposed the restrictions $\wt{r}'\leq\wt{q}'$ and $q\leq r$ and used the fact that $\wt{q}'=q$.

Suppose now that $r>r_1$ and $\wt{r}>\wt{r}_1$. To obtain (\ref{est:inhomogeneous strichartz for wave sharp}), simply apply Besov embedding (Lemma \ref{lem:besov embedding}) to the result obtained for the case when $r=r_1$ and $\wt{r}=\wt{r}_1$.
\end{proof}

\begin{remark}
In all the Strichartz estimates given in this section, one may exchange the infinite time interval $\R$ appearing in the spacetime norms with a finite time interval $I$ or $J$. This is done by redefining each $U(t)$ as $1_I(t)U(t)$, where $1_I$ is the characteristic function of $I$ on $\R$, and by redefining $F$ as $1_JF$.
\end{remark}

\begin{remark}\label{rem:KGE}
When Theorem \ref{th:inhomogeneous} (i) is applied to the inhomogeneous Klein--Gordon equation
\[-u''(t)+\Delta u(t)-u=F(t),\qquad u(0)=u''(0)=0,\qquad t\geq0,\]
the range of inhomogeneous Strichartz estimates given by Nakamura and Ozawa \cite[Proposition 2.1]{NO01} is slightly improved. These estimates are obtained in a manner analogous to the wave equation using Besov space norms, rather than homogeneous Besov space norms. For the required dispersive estimate, see \cite[pp.~261--262]{NO01}. A precise statement of our corollary and other details of its proof are available in \cite[Section 5.8]{rT08}.
\end{remark}

\bigskip

\noindent\textit{Acknowledgements.} The author would like to thank Michael Cowling for introducing him to the field of Strichartz estimates and Andrew Hassell for the interest he took in this work.

\providecommand{\bysame}{\leavevmode\hbox to3em{\hrulefill}\thinspace}
\providecommand{\MR}{\relax\ifhmode\unskip\space\fi MR }
\providecommand{\MRhref}[2]{%
  \href{http://www.ams.org/mathscinet-getitem?mr=#1}{#2}
}
\providecommand{\href}[2]{#2}

\end{document}